\documentclass[12pt,a4paper,reqno]{amsart}
\numberwithin{equation}{section}
\usepackage[utf8]{inputenc}
\usepackage[english]{babel}
\usepackage{hyperref}

\usepackage{multirow} 
\usepackage{graphicx}
\usepackage{array}
\usepackage{longtable}
\usepackage{tikz-cd}
\newtheorem{theo}{Theorem}[section]
\newtheorem{lem}{Lemma}[section]
\newtheorem{defi}{Definition}[section]

\newtheorem{prop}{Proposition}[section]
\newtheorem{exm}{Example}[section]

\usepackage[foot]{amsaddr}
\usepackage{url}
\usepackage{hyperref}

\usepackage{amsmath}
\usepackage{amssymb, amscd}
\usepackage[all, cmtip]{xy}
\usepackage{xcolor}
\usepackage{multirow} 
\usepackage{longtable}
\usepackage{array}

\theoremstyle{definition}

{\sc}

{}
{}
{\sc}
\newtheorem{remark}{Remark}[section]

\newcommand{\thismonth}{\ifcase\month\or
  January\or February\or March\or April\or May\or June\or
  July\or August\or September\or October\or November\or December\fi
  \space\number\year}

\makeatletter
\newcommand{\rssymb}[2]{\newcommand{#1}{{\mathrmsl{#2}}}}
\newcommand{\calsymb}[2]{\newcommand{#1}{{\mathcal{#2}}}}
\newcommand{\bbsymb}[2]{\newcommand{#1}{{\mathbb{#2}}}}
\newcommand{\lieoper}[2]{\newcommand{#1}{\mathop
  {\mathfrak{#2}\null}\nolimits}}
\newcommand{\oper}[3][n]{\newcommand{#2}{\mathop
  {\mathrm{#3}\null}\ifx n#1\nolimits\else\limits\fi}}
\newcommand{\rsoper}[3][n]{\newcommand{#2}{\mathop
  {\mathrmsl{#3}\null}\ifx n#1\nolimits\else\limits\fi}}
\bbsymb\C{C} \bbsymb\F{F} \bbsymb\HQ{H}\bbsymb\N{N} \bbsymb\Q{Q}
\bbsymb\R{R} \bbsymb\U{U} \bbsymb\V{V} \bbsymb\W{W} \bbsymb\Z{Z}
\bbsymb\bbf{F} \bbsymb\bbk{K} \bbsymb\bbi{I} \bbsymb\bbl{L}
\bbsymb\bbo{O} \bbsymb\bbj{J} \bbsymb\bby{Y} \bbsymb\bbp{P}
\bbsymb\bba{A}
\calsymb\cA{A} \calsymb\cB{B} \calsymb\cC{C} 
\calsymb\cM{M} \calsymb\cN{N} \calsymb\cO{O} \calsymb\cP{P}
\calsymb\cU{U} \calsymb\cV{V} \calsymb\cW{W} \calsymb\cX{X}
\calsymb\cY{Y} \calsymb\cZ{Z}
\renewcommand{\geq}{\geqslant} \renewcommand{\leq}{\leqslant}
\oper\End{End} \oper\Hom{Hom}                    
\oper\Sym{Sym} \oper\Skew{Skew}
\oper\Aut{Aut}                                   
\oper\GL{GL} \oper\SL{SL}\oper\Symp{Sp} \oper\CO{CO} \oper\On{O}
\oper\SO{SO} \oper\Pin{Pin} \oper\Spin{Spin} \oper\CU{CU}
\oper\Un{U} \oper\SU{SU} \oper\PSU{PSU} \rsoper\Diff{Diff}
\rsoper\SDiff{SDiff}
\lieoper\der{der}                                
\lieoper\gl{gl} \lieoper\sgl{sl}\lieoper\symp{sp} \lieoper\co{co}
\lieoper\so{so} \lieoper\spin{spin} \lieoper\cu{cu} \lieoper\un{u}
\lieoper\su{su} \rsoper\Vect{Vect} \rsoper\Ham{Ham}
\def\la#1{\hbox to #1pc{\leftarrowfill}}
\def\ra#1{\hbox to #1pc{\rightarrowfill}}

\newcommand{\Norm}[2][]{\bigl|\mkern-3mu\bigr|#2\bigr|\mkern-3mu\bigr|
  _{\lower1pt\hbox{${}_{#1}$}}}

\rsoper\dimn{dim}                           
\rsoper\grad{grad}                          
\rsoper\kernel{ker}\rsoper\image{im}        
\rsoper\alt{alt}   \rsoper\sym{sym}         
\rsoper\Ad{Ad}     \rsoper\ad{ad}           
\rsoper\CoAd{CoAd} \rsoper\coad{coad}       
\rsoper\trace{tr}  \rsoper\trfree{tf}       
\rsoper\detm{det}                           
\rsoper\Vol{Vol}                            
\rsoper\divg{div}                           
\rsoper\sign{sign}                          
\rssymb\iden{id}                            
\rssymb\vol{vol}                            
\oper\Imag{Im}\oper\Real{Re}                
\newcommand{\sd}{{\raise1pt\hbox{$\scriptscriptstyle +$}}}
\newcommand{\asd}{{\raise1pt\hbox{$\scriptscriptstyle -$}}}
\newcommand{\sdasd}{{\raise1pt\hbox{$\scriptscriptstyle\pm$}}}

\newcommand{\asdsd}{{\raise1pt\hbox{$\scriptscriptstyle\mp$}}}

\rsoper\scal{scal}
\def\kahl/{k\"ahler}
\def\Kahl/{K{\"a}hler}

\begin{document}

\title[Exceptional  Fano 3-folds from  rational curves]
{Exceptional Fano 3-folds from rational curves}

\author[J. Cuadros]{Jaime Cuadros Valle$^1$}
\author[J. Lope]{Joe Lope Vicente$^1$}

\address{$^1$Departamento de Ciencias, Secci\'on Matem\'aticas,
Pontificia Universidad Cat\'olica del Per\'u,
Apartado 1761, Lima 100, Per\'u}
\email{jcuadros@pucp.edu.pe}
\email{j.lope@pucp.edu.pe}

\date{\thismonth}

\begin{abstract} 
We show exceptionality of certain  families  of  non-quasismooth  weighted hypersurfaces. In particular these admit  Kähler-Einstein metrics. Our examples  are produced by the monomials generating the complex deformations of orbifolds whose corresponding  $S^1$-Seifert bundles are smooth  rational homology 7-spheres admitting Sasaki-Einstein metrics.  From our construction, it follows that these exceptional  Fano hypersurfaces describe elements in the boundary of the  K-moduli of   
$\mathbb{Q}$-Fano 3-folds.
\end{abstract}

\maketitle

\medskip

\medskip

\maketitle
\vspace{-2mm} 


\section{Introduction}

The existence of Kähler-Einstein metrics on Fano manifolds is one of the main branches of research in complex 
and differential geometry, but singular Fano varieties have also received significant interest. The $\alpha$-invariant, introduced by Tian in \cite{Ti}, provides a  criterion for the existence of this type of metrics on 
a compact Fano manifold. Later, Demailly and Kollár  generalize this invariant for Fano varieties with quotient singularities \cite{DK}. Because the existence of Kähler-Einstein metrics on orbifolds is equivalent to the existence of Sasaki-Einstein metrics on the corresponding smooth $S^1$-Seifert bundle \cite{BG1}, this generalized version of the  $\alpha$-invariant was intensely used in the last two decades (see for instance  \cite{BGN2,  BGK, Kol1, Kol2}). Furthermore, Demailly reformulated this invariant in terms of the log canonical threshold \cite{CS} and made it suitable for the study of Kähler-Einstein metrics on more general Fano varieties.   Originally, the $\alpha$-invariant was designed  to determine  whether the Fano manifold was K-stable which is  an algebro-geometric condition introduced by Tian \cite{Ti} ({\it c.f.} \cite{Do}) in order to  characterize  the existence of a Kähler–Einstein metric on a Fano manifold  by this condition.  This correspondence, called the Yau-Tian-Donaldson conjecture, was  proven by Chen, Donaldson and Sun in \cite{CDS} and Tian \cite{Ti2}. In recent years, a more algebraic approach to K-stability has permitted  a generalization of this correspondence which was proven in \cite{LXZ} for singular Fano varieties.

In this note, we construct  singular Fano 3-folds  that are  exceptional, that is, varieties with $\alpha$-invariant greater than 1 \cite{Bi2}.  Since Odaka and Sano proved that Fano varieties of dimension $n$ with at most Kawamata log-terminal singularities (klt for short) and whose $\alpha$-invariant is  greater than $\frac{n}{n+1}$ must be K-stable  \cite{OS}, it follows from the generalization of the Yau-Tian-Donaldson conjecture for singular varieties \cite{LXZ} that the  examples exhibited here  admit weak K\"ahler-Einstein metrics.  

These examples are constructed from deformations of  hypersurfaces  cut out by certain invertible polynomials. 
 Our  perturbations  determine  varieties   that are not  quasismooth, so their singularities are not always quotient singularities but we showed that all of them are klt using a criterion given by \cite{IP} suitable for  hypersurfaces  embedded  in toric varieties.  Unlike the quasismooth case,   establishing the  exceptionality of non-quasismooth hypersurfaces, turns out to be quite involved and we do this using  similar ideas as the ones exposed in \cite{To}, where Totaro generalized the methods of Johnson and Kóllar \cite{JK2} replacing the  usual multiplicity of a variety at a point with the notion of  weighted multiplicities instead. 
 
 Our examples  are built up from small deformations of  Thom-Sebastiani sums of two blocks of invertible polynomials 
\begin{itemize}
\item $f=z_0^{a_0}+z_1^{a_1}+z_4 z_2^{a_2}+z_2 z_3^{a_3}+z_3 z_4^{a_4}$
\item  $f=z_0^{a_0}+z_0 z_1^{a_1}+z_4 z_2^{a_2}+z_2 z_3^{a_3}+z_3 z_4^{a_4}$
\item $f=z_1 z_0^{a_0}+z_0 z_1^{a_1}+z_4 z_2^{a_2}+z_2 z_3^{a_3}+z_3 z_4^{a_4}$.
\end{itemize}
such that  the weighted hypersurface cut out  by these  invertible polynomials have rational homology groups equal to that of  a complex projective 3-spaces and such that the cycle block (sometimes called loop block) $z_4 z_2^{a_2}+z_2 z_3^{a_3}+z_3 z_4^{a_4}$  appearing in  these sums,  cut out a rational curve, see Remark 2.1. These requirements allow us to successfully compute the bounds for the $\alpha$-invariant for the perturbations of $f$.  

The exceptional non-quasismooth Fano hypersurfaces described above are limiting points of the  moduli of  quasismooth hypersurfaces  admitting Kähler-Einstein metrics and moreover, give rise to non-smooth links whose metric cones may be indicative of degenerating Calabi-Yau cones \cite{Od, Li}. 

The arguments shown in this paper can be generalized to higher dimensions and we do this in a succeeding work.

%
%
%
%

\smallskip

The paper is organized as follows. Section 2 reviews the background material relevant for this paper. In Section 3 we proved that our examples are klt. In Section 4 we establish the exceptionality of our families and  give explicit examples.

\section{Preliminaries}

\subsection{Weighted hypersurfaces}

 As an algebraic variety, a weighted projective space can be defined as  $\displaystyle{\mathbb{P}\left(w_0, \ldots, w_{n}\right)=\operatorname{Proj}(S(\mathbf{w})),}$  
where $S(\mathbf{w})=\mathbb{C}\left[x_0, \ldots, x_{n+1}\right]$ is the graded polynomial ring such that the weight of each $x_i$ equals $w_i,$ 
where $\mathbf{w}=\left(w_0, \ldots, w_n\right)$ is a sequence of positive integers.  However it is useful to give the orbifold construction of this set: consider  the weighted $\mathbb{C}^*$-action on the affine space $\mathbb{C}^{n+1},$ defined by
$$
\left(z_0, \ldots, z_n\right) \mapsto\left(\lambda^{w_0} z_0, \ldots, \lambda^{w_n} z_n\right).
$$
Then we obtain the weighted projective space with a canonical orbifold structure,  defined as the quotient space $$Y=\mathbb{P}\left(w_0, \ldots, w_{n}\right)=\left(\mathbb{C}^{n+1}-\{\mathbf{0}\}\right) / \mathbb{C}^*.$$ 

In section 4, we consider  the weighted projective space as the quotient stack 
$$\mathcal{Y}= \left[\left(\mathbb{C}^{n+1}-\{\mathbf{0}\}\right) / \mathbb{C}^* \right]$$ instead.  
Here $\mathcal{Y}$ is a smooth Deligne-Mumford stack with canonical class $K_{\mathcal{Y}}=O_{\mathcal{Y}}\left(-\sum w_i\right)$. We say that $Y$ is well-formed if the stack $\mathcal{Y}$ has trivial stabilizer in codimension 1, or equivalently if $\operatorname{gcd}\left(w_0, \ldots, \widehat{w}_i, \ldots, w_{n}\right)=1$ for each $i.$ Here the hat symbol means delete that corresponding element. In the well-formed case, the formula for the canonical class of the variety $Y$ coincides with the  formula given for the stack. 

The grading of $S(\mathbf{w})=\mathbb{C}\left[x_0, \ldots, x_{n+1}\right]$ of the polynomial ring defining the well-formed projective space $Y$ determines line bundles $\mathcal{O}(m)$ on the weighted projective stack $\mathcal{Y}.$ However on the corresponding coarse weighted projective  space  $Y$, the sheaf $\mathcal{O}_Y(m)$  is only the reflexive sheaf associated to a Weil divisor in general.   Actually, the divisor class $O_Y(m)$ is Cartier if and only if $m$ is a multiple of every weight $w_i$.
The intersection number $\int_{\mathcal{Y}} c_1(\mathcal{O}(1))^n$ is $1 /\left(w_0 \cdots w_n\right)$. More generally, for an integral closed substack $Z$ of dimension $r$ in $\mathcal{Y}$, its degree is given by   $\int_Z c_1(\mathcal{O}(1))^r$.  

For  closed substacks defined by hypersurfaces  of degree $d$ in $\mathcal{Y}$, that is,  defined by  weighted-homogeneous polynomials of degree $d$
one  can be more precise. Recall  that a polynomial $f \in \mathbb{C}\left[z_0, \ldots, z_n\right]$ is said to be a weighted homogeneous polynomial of degree $d$ and weight vector $\mathbf{w}=$ $\left(w_0, \ldots, w_n\right)$,  if for any $\lambda \in \mathbb{C}^{*}$
$$
f\left(\lambda^{w_0} z_0, \ldots, \lambda^{w_n} z_n\right)=\lambda^d f\left(z_0, \ldots, z_n\right).$$
From the affine algebraic variety  $V_f=\{f=0\} \subset \mathbb{C}^{n+1}$ one construct the  weighted hypersurface 
$$Z_f=(V_f-\{\mathbf{0}\})/ \mathbb{C}^* \subset Y=\mathbb{P}(\mathbf{w}).$$ 
As a substack in $\mathcal{Y},$ the degree of $Z_f$ is $\displaystyle{Z_f\cdot c_1(\mathcal{O}(1))^{n-1}= \frac{d}{w_0 \ldots w_n}}.$ 

If the cone $V_f$ is smooth everywhere except at the origin in $\mathbb{C}^{n+1}$ one says that  $Z_f$  is quasismooth. Notice that  quasismooth weighted hypersurfaces  have  only cyclic quotient singularities and hence are klt. The weighted hypersurface is well-formed if $Y=\mathbb{P}\left(w_0, \ldots, w_{n}\right)$ is well-formed and  the intersection of  $Z_f$ with the singular set of $Y$, that is the locus where the stack $\mathcal{Y}$ has non-trivial stabilizer,  has codimension at least 2 in $Z_f$. We define the index $I$ of the weighted hypersurface  degree $d$   as $I=\sum w_{i}-d.$  When $Z_f$ is well-formed, the canonical divisor satisfies the adjunction formula $K_{Z_f}=\mathcal{O}_{Z_{f}}\left(d-\sum w_{i}\right).$
 
 For well-formedness of the weighted variety we have the following criterion \cite{IF}: 
 \begin{lem}
 A hypersurface defined by the weighted homogeneous polynomial $f$ of degree $d$ is well-formed in the well-formed weighted projective space $Y=\mathbb{P}\left(w_0, \ldots, w_n\right)$ if 
 $\operatorname{gcd}\left(w_0, \ldots, \hat{w}_i, \ldots, \hat{w}_j, \ldots, w_n\right) \mid d$
for distinct $i, j=0, \ldots, n$.  
\end{lem}

 There are  conditions that determine when a specific polynomial determines a quasismooth weighted hypersurface \cite{IF,JK1}: 
 
 \begin{lem}
 A weighted hypersurface of degree $d$ in $\mathbb{P}\left(w_0, \ldots, w_4\right)$, where $d>a_i$, is quasismooth if and only if the following hold:
\begin{enumerate}
\item For each $i=0, \cdots, 4$ there is a $j$ and a monomial $z_i^{m_i} z_j$ of degree $d$. Here $j=i$ is possible.
\item For all distinct $i, j$ either there is a monomial $z_i^{b_i} z_j^{b_j}$ of degree $d$ or there exist monomials $z_i^{n_1} z_j^{m_1} z_k, z_i^{n_2} z_j^{m_2} z_l$ of degree $d$ with $\{k, l\} \neq\{i, j\}$ and $k \neq l$.
\item For every $i, j$ there exists a monomial of degree $d$ that does not involve either $z_i$ or $z_j$.
\end{enumerate}
\end{lem}
\smallskip

In \cite{Ti} Tian showed that a  Fano manifold of dimension $n$  with $\alpha$-invariant  greater
than $\frac{n}{n+1}$ admits a Kähler–Einstein metric. A generalization of this result for Fano varieties with
quotient singularities was given by Demailly and Koll\'ar (see \cite{DK}, Criterion 6.4). For the particular case of a Fano variety given as 
a  weighted hypersurface is quasismooth, this criterion gives  the  following estimate, whose proof can be found in \cite{BBG}, Corollary 5.4.8 ({\it c.f.}  \cite{BGK}).

\begin{theo}
Let $Z_f \subset \mathbb{P}\left(w_0, \ldots w_n\right)$ be a quasismooth weighted homogeneous Fano hypersurface of degree $d$. Then $Z_f$ admits a Kähler-Einstein orbifold metric if the following estimate holds:
$$
d I<\frac{n}{(n-1)} \min _{i, j}\left\{w_i w_j\right\}.
$$
\end{theo}

One interesting  feature of the previous theorem is related with the important problem in differential geometry: the search for Einstein metrics on odd dimensional manifolds. Indeed,  recall \cite{Mi} that for a  weighted homogeneous  polynomial $f \in \mathbb{C}\left[z_0, \ldots, z_n\right],$  the link $L_f$ of the singularity is defined to be 
$$
L_f=V_f \cap \mathbb{S}_{\varepsilon},
$$
with  $\mathbb{S}_{\varepsilon}$  a sufficiently small $(2n+1)$-sphere centered at $\bf{0}$. For the quasismooth case it  is well-known that
the link  is a smooth ($2n-1$)-dimensional manifold and moreover,  we have the following correspondence: the manifold $L_f$ admits Sasaki-Einstein metric $g$ if and only if the Fano orbifold  $(V_f-\{ 0\})/\mathbb{C}^*=Z_f$ admits a Kähler-Einstein orbifold metric $h$ \cite{BBG} and furthermore, the metric cone  $(Y,\bar{g})=(L_f \times \mathbb{R}^{+}, \bar{g}=d r^2+r^2 g)$  is  Ricci-flat.  

On the other hand, in the  non-quasismooth case the Milnor fibration is still valid but  the link $L_f$ is no longer smooth, since  non-quasismooth klt singularties are  key to understand the limits of K-moduli of Fano varieties and its extension to moduli of K-stable Fano cones \cite{Od},  it is natural to look for a similar or limiting  correspondence to the one given above for the quasismooth case. 
 However, to determine whether there are  singular  Kähler-Einstein metrics on  non-quasismooth hypersurfaces proves to be complicated and we appeal to the weighted cone construction developed in \cite{To} to do this  for certain hypersurfaces close enough to quasismooth hypersurfaces admitting Kähler-Einstein metrics, see Section 4.

\subsection{Monomials of interest}

Our results in the next sections are based on non-quasismooth polynomials that are designed from the monomials defining the weighted homogeneous polynomials of certain degrees $d$ and weight vectors $\mathbf{w}$. In order to do so we use previous results given in \cite{CL2} on the explicit description  of the monomials generating  $H^0(\mathbb{P}(\mathbf{w}), \mathcal{O}(d))$ for certain invertible polynomials. 
Recall that  invertible polynomials are polynomials of the form $f=\sum_{i=1}^n \prod_{j=1}^n z_j^{a_{i j}}$, where $A= \left(a_{i j}\right)_{i, j=1}^n$ is a non-negative integer-valued matrix which is invertible over $\mathbb{Q}$ and where $f$ is quasihomogeneous.  Due to the Kreuzer-Skarke classification of invertible polynomials \cite{KS} we know that any invertible polynomial, up to permutation of variables, can be written as a Thom-Sebastiani sum of three types of polynomials usually called atoms:
\begin{itemize}
\item Fermat type: $w=z^a$,
\item  Chain type: $w=z_1^{a_1} z_2+z_2^{a_2} z_3+\ldots+z_{n-1}^{a_{n-1}} z_n+x_n^{a_n}$, and
\item  Loop or cycle type: $w=z_1^{a_1} z_2+z_2^{a_2} z_3+\ldots+z_{n-1}^{a_{n-1}} z_n+z_n^{a_n} z_1$. 
\end{itemize}

We consider the following types of invertible  polynomials, which are Thom-Sebastiani sums of a cycle polynomial $z_{4}z_{2}^{a_{2}}+z_{2}z_{3}^{a_{3}}+z_{3}z_{4}^{a_{4}} $ and some other atomic block (Fermat, chain or cycle) in terms of $z_0$ and $z_1:$  
\begin{align*}
    \mbox{Type I}: & \  f =z_{0}^{a_{0}}+z_{1}^{a_{1}}+z_{4}z_{2}^{a_{2}}+z_{2}z_{3}^{a_{3}}+z_{3}z_{4}^{a_{4}} \\
    \mbox{Type II}: & \  f =z_{0}^{a_{0}}+z_{0}z_{1}^{a_{1}}+z_{4}z_{2}^{a_{2}}+z_{2}z_{3}^{a_{3}}+z_{3}z_{4}^{a_{4}} \\
    \mbox{ Type III}: & \  f =z_{1}z_{0}^{a_{0}}+z_{0}z_{1}^{a_{1}}+z_{4}z_{2}^{a_{2}}+z_{2}z_{3}^{a_{3}}+z_{3}z_{4}^{a_{4}}.  
\end{align*}
We  assume that the weight vectors associated to these families of polynomials  satisfy certain pack of conditions on the weights and degrees so that the  weighted hypersurface $X_f\subset \mathbb{P}(\bf{w})$ behaves nicely: are well-formed rational homology projective 3-spaces.  These conditions given below,  are labeled as conditions (2.1):
$$
\left \{\begin{array}{l}
\mathbf{w}=\left(w_0, w_1, w_2, w_3, w_4\right)=\left(m_3 v_0, m_3 v_1, m_2 v_2, m_2 v_3, m_2 v_4\right).\\
\gcd(m_{2},m_{3})=1.\\
 d=m_2m_3. \\
\gcd(v_{0},v_{1})=1\\  
\gcd(v_{i},v_{j})=1,  i\neq j \hbox { with } i,j\in\{2,3,4\}. \\
m_3=a_2a_3a_4+1.\\
\end{array}\right .
$$


\begin{remark} 
From the conditions (2.1) on the weights of the  polynomials given above one can extract the cycle block $g=z_{4}z_{2}^{a_{2}}+z_{2}z_{3}^{a_{3}}+z_{3}z_{4}^{a_{4}}$ of three variables which cuts out  a quasismooth rational curve  $X_{m_3}\subset \mathbb{P}(v_2,v_3, v_4)$ of degree $m_3.$  
Indeed, consider the matrix of exponents of $g$ given by $A=\left(a_{i j}\right)_{i, j}.$  The map $$\phi_A: \mathbb{P}(v_2,v_3, v_4)\rightarrow  \mathbb{P}^2$$ 
given by 
$$\left(x_2: x_3: x_4\right) \stackrel{\phi_A}{\longmapsto}\left(y_2: y_3: y_4\right), \quad y_j=\prod_{i=2}^4 x_i^{a_{i j}}.$$  
In \cite{Ko2}, Kollár  shows  that $\phi_A$ maps $\mathbb{P}(v_2,v_3, v_4)$ birationally to $\mathbb{P}^{2}$ and so $X_{m_3}$ is mapped birationally to the  hyperplane 
$\left \{ y_2+y_3+y_4=0\right \}\subset \mathbb{P}^{2}$. Thus $X_d$ is a rational curve.  A key in the argument presented in this reference (that works the details for arbitrary dimensions) is  that  the determinant  of $A $ equals the degree $m_3$, that is, $a_2a_3a_4+1=m_3.$ 
 \end{remark}

\begin{lem} Assume that $f$ is of type I, II or II and its  weight vector satisfies the conditions given in (2.1), then the orbifold $X_f$ cut out by $f$ is a rational homology projective 3-space. 
\end{lem}
\begin{proof} Here we only give a sketch of this proof since the main arguments are presented elsewhere. The idea is to first  compute the cohomology groups of the links 
$L_f$ of the isolated singularities of $f$ using the Milnor-Orlik formula \cite{MO} for weighted polynomials and conclude that the links are rational homology spheres and this was done in  \cite{BGN1} ({\it c.f.} \cite{CL}). Since the link $L_f$ is a Seifert $S^1 $-bundle over the orbifold $X_f \subset \mathbb{P}\left(w_0, \ldots, w_4\right)$ the Leray spectral sequence 
given by 
$H^i\left(X_f, R^j \pi_* \mathbb{Q}_{L_f}\right) \Rightarrow H^{i+j}(L_f, \mathbb{Q})$
gives \cite{OW, Ko4}: 
\begin{itemize}
\item  $\operatorname{dim} H^i(L_f, \mathbb{Q})=\operatorname{dim} H^i(X_f, \mathbb{Q})-\operatorname{dim} H^{i-2}(X_f, \mathbb{Q})$ for $i \leq \operatorname{dim} X_f$,
\item  $\operatorname{dim} H^{i+1}(L_f, \mathbb{Q})=\operatorname{dim} H^i(X_f, \mathbb{Q})-\operatorname{dim} H^{i+2}(X_f, \mathbb{Q})$ for $i \geq \operatorname{dim} X_f$.
\end{itemize}
Thus, $L_f$ is a rational homology 7-sphere if and only if $X_f$ is a rational homology complex projective 3-space. The lemma follows from this observation.
\end{proof}

In \cite{CL2} we showed the following result: 
\begin{theo} 
Assume that $f$ and its  weight vector are as  in the previous lemma, then 
 \begin{itemize}
\item    If $f$ is a polynomial of type I, then 
$$H^0(\mathbb{P}(\mathbf{w}), \mathcal{O}(d))=Span \left\{z_1^{m_2}, z_0 z_1^{m_2-1}, \ldots, z_0^{m_2-1} z_1, z_0^{m_2}, z_4 z_2^{a_2}, z_2 z_3^{a_3}, z_3 z_4^{a_4}\right\}.$$ 

\item If $f$ is a polynomial of type II and its associated weight vector $\bf w$ does not admit polynomial of type I, then $$H^0(\mathbb{P}(\mathbf{w}), \mathcal{O}(d))=Span \left \{z_{0}^{m_{2}-kv_{1}}z_{1}^{k},   z_{4}z_{2}^{a_{2}},   z_{2}z_{3}^{a_{3}},   z_{3}z_{4}^{a_{4}},\mbox{ where }  0\leq k  \leq \left\lfloor \dfrac{m_{2}}{v_{1}}\right \rfloor \right \} .$$

\item If $f$ is a polynomial of type III and its associated weight vector $\bf w$ does not admit polynomial of type II, then 
$$H^0(\mathbb{P}(\mathbf{w}), \mathcal{O}(d))=Span \left \{z_{0}^{a_0-kv_{1}}z_{1}^{1+kv_0},   z_{4}z_{2}^{a_{2}},   z_{2}z_{3}^{a_{3}},   z_{3}z_{4}^{a_{4}}, \mbox{ where }  0\leq k  \leq \left\lfloor \dfrac{m_{2}}{v_{0}v_{1}}\right \rfloor\right \}.$$ 
\end{itemize}
\end{theo}
\smallskip

For the weight vector $\textbf{w}$ given above, we consider  hypersurfaces $X\subset\mathbb{P}(\textbf{w})$ defined by  equations of the form 
\begin{equation}    
X:f=h(z_{0},z_{1})+z_{4}z_{2}^{a_{2}}+z_{2}z_{3}^{a_{3}}+z_{3}z_{4}^{a_{4}}=0
\end{equation}
where $h(z_{0},z_{1})$ are  binomials formed by elements in $H^0(\mathbb{P}(\mathbf{w}), \mathcal{O}(d))$ chosen in such a way that the hypersurface $X$ is  not quasismooth.   We find conditions on the exponents of $z_{0}$ and $z_{1}$ in $h$ so that the subvariety $X$ is klt and exceptional.


\subsection{Singularities of the minimal model program}
For  precise definitions of the singularities of the minimal model appearing in this section and their possible relations see \cite{Ko3, KM}.   For sake of completeness, below we give  definitions of the singularities appearing in this article: klt,  canonical and log canonical singularties but we would rather give them in terms of the notion of minimal log discrepancies developed by  Shokurov \cite{Sh} 

 Consider the log pair $(X, D),$ with  $X$  a normal complex variety and $D$ an effective $\mathbb{Q}$-divisor on $X$ such that $K_X+D$ is $\mathbb{Q}$-Cartier, where $K_X$ denotes  the canonical divisor of $X.$ 
Let  $$\sigma: Y \rightarrow X$$ be a  proper birational morphism from another normal variety $Y$ and $E \subset Y$ be a prime divisor. If $\omega$ is a top rational form on $X$ defining $K_X$, we consider $K_Y$ the canonical divisor defined by $\sigma^*\omega.$  

\begin{defi} The log discrepancy of $E$ with respect to the morphism 
$\sigma$ as above is defined as: 
$$
a_E(X, D):=\operatorname{ord}_E\left(K_{Y}+E-\sigma^*\left(K_X+D\right)\right) =1-\operatorname{ord}_E\left (\sigma^*\left(K_X+D\right)-K_Y\right ).$$
 \end{defi}
 
 The $\log$ discrepancy depends only on the discrete valuation defined by $E$ on $k(X)$, in particular it is independent on the extraction $Y$ where $E$ appears as a divisor.

\begin{defi}
 For any point $x$ of the scheme $X$,  the minimal log discrepancy at $x$ is defined by  
$$
\operatorname{mld}(X, D, x):=\inf_{\sigma(E)=\bar{x}} \left\{a_E(X, D)\right\},
$$
where the infimum is taken over  all prime divisors $E$ on proper birational morphisms $\sigma: Y \rightarrow X.$ 
\end{defi}
If this infimum is non-negative one says that the pair $(X,D)$ has a log canonical singularities at $x.$ In this case, as explained in \cite{Ko3}, by Hironaka,  there exists a proper birational morphism $\sigma: Y \rightarrow X$ such that $Y$ is nonsingular, $\sigma^{-1}(\bar{x})$ is a divisor on $Y$, and there exists a simple normal crossings divisor $\sum_i E_i$ on $Y$ which supports both $\sigma^{-1}(\bar{x})$  and $K_{Y}-\sigma^*\left(K_X+D\right)$. Then

$$
\operatorname{mld}(X, D, x)=\min _{\sigma\left(E_i\right)=\bar{x}} a_{E_i} \left( X, D\right).
$$

\begin{defi}
 The global minimal log discrepancy of the variety $X$ is defined as the set 
$$
\operatorname{mld}(X, D):=\inf _{x \in X} \operatorname{mld}(X, D, x) .
$$
where the infimum  is taken over all codimension $\geq 1$ points $x$ on $X$ (recall that the codimension of a point is defined to be the codimension of its closure). The pair $(X, D)$ as above is:
\begin{itemize}
\item {canonical}  if $\operatorname{mld}(X, D) \geq 1$ for every point $x \in X$ of codimension at least 2 . 
\item  {log canonical (lc)} if $\operatorname{mld}(X, D) \geq 0$. 
\item   {Kawamata log terminal (klt)} if $\operatorname{mld}(X, D)>0.$
\end{itemize}
\end{defi}

Since quotient singularities are klt (see \cite{Ko3}, Section 3.2), quasismooth varieties are klt. However, proving the klt property for the non-quasismooth varieties that we manufacture  as deformations of quasismooth hypersurfaces  turns out to be  subtle. 
\medskip


The following remark will be used implicitly in Section 4:
\begin{remark} As pointed out in \cite{To}, if the stabilizer groups are trivial in codimension 1,  one does not need to distinguish between the property of being klt or lc for a normal Deligne-Mumford stack and for its associated coarse moduli space; this follows from the fact that for a pair $(X, D),$ these properties  do not change under finite coverings of normal varieties which are étale in codimension 1 \cite{Ko2} Corollary 2.43.  
\end{remark}

\subsection{Exceptionality.}  

Recall that for a given  divisor $D$ on a variety $X$, we write $D\sim_{\mathbb{Q}} -K_{X}$ when the divisor $D$ is $\mathbb{Q}$-linearly equivalent to $-K_{X}$. Thus, we define the set:
$$\vert -K_{X}\vert_{\mathbb{Q}}=\{ D \mbox{ is }  \mathbb{Q}\mbox{-divisor} / \ D \mbox{ is effective and } D\sim_{\mathbb{Q}} -K_{X}\} $$
 \begin{defi}
     A klt Fano variety $X$ is exceptional if for any $\mathbb{Q}$-divisor $D\in |-K_{X}|$, the pair $(X,D)$ is klt.
 \end{defi}
Another way to describe the exceptionality of a variety $X$ is through the notion of \textit{global log canonical threshold}. 

\begin{defi} For an effective $\mathbf{Q}$-Cartier $\mathbf{Q}$-divisor $D$ on a klt Fano variety $X$, the $\log$ canonical threshold $\operatorname{lct}(X, D)$ is the supremum of the real numbers $\lambda$ such that the pair $(X, \lambda D)$ is log canonical (lc). The global log canonical threshold of $X$ is  the real number
    $$\operatorname{glct}(X)=\sup\{\lambda\in\mathbb{R} : (X,\lambda D) \mbox{ is log canonical  for all }D\in|-K_{X}|\}.$$
\end{defi}

We have the following equivalence which is due to Birkar \cite{Bi2}:
\begin{lem}
Let $X$ be a klt Fano variety. Then $X$ is exceptional if and only if $\operatorname{glct}(X)>1$.
\end{lem}


\section{klt weighted Fano 3-folds}
In this section, we study the singularities of certain perturbations of invertible polynomials that turn out to be non-quasismooth. More precisely, we show that these perturbations are klt (even though they are not quotient singularties). In order to do this, we review a  criterion studied by Ishii and Prokohorov \cite{IP}, suitable  for hypersurfaces embedded in toric varieties, which  determines whether the singularities in the set of zeroes of a polynomial cutting out the weighted hypersurface  are canonical. But first we need a technical lemma  designed for weighted hypersurfaces  with certain 
numerical constraints whose  corresponding links are rational homology 7-spheres (\cite{BGN1, CL}).  

\begin{lem}
    Given a weight vector ${\bf w}=(w_{0},w_{1},w_{2},w_{3},w_{4})=(m_{3}v_{0},m_{3}v_{1},m_{2}v_{2},m_{2}v_{3},m_{2}v_{4}) $ associated with a quasi-homogeneous  polynomial $f$ of type I,II or III of degree $d=m_{2}m_{3}$. If $m_{3}=a_{2}a_{3}a_{4}+1$, then the integers  
    $$\tilde{w}_{2}=m_{2}(a_{4}a_{3}-a_{4}+1), \ \ \tilde{w}_{3}=m_{2}(a_{2}a_{4}-a_{2}+1) \ \ \mbox{ and } \ \ \tilde{w}_{4}=m_{2}(a_{2}a_{3}-a_{3}+1)$$
    satisfy the equalities:
    \begin{center}
$\begin{cases}
	a_{2}\tilde{w}_{2}+\tilde{w}_{3}=d\\
	a_{3}\tilde{w}_{3}+\tilde{w}_{4}=d \\
    a_{4}\tilde{w}_{4}+\tilde{w}_{2}=d
 \end{cases}$
\end{center}
Moreover, these verify $\tilde{w}_{2}+\tilde{w}_{3}+\tilde{w}_{4}=w_{2}+w_{3}+w_{4}$.
\end{lem}
\begin{proof}
    As $d=m_{2}m_{3}$ and $m_{3}=a_{2}a_{3}a_{4}+1$, we have 
    $$a_{2}\tilde{w}_{2}+\tilde{w}_{3}=m_{2}\left(a_{2}(a_{4}a_{3}-a_{4}+1)+a_{2}a_{4}-a_{2}+1\right)=m_{2}m_{3}=d.$$
    Similarly one can prove the other equalities. On the other hand, since each of  the monomials $z_{4}z_{2}^{a_{2}}, z_{2}z_{3}^{a_{3}}$ and $z_{3}z_{4}^{a_{4}}$ have degree $d$, it follows that  $w_{2}=m_{2}(a_{4}a_{3}-a_{3}+1)$, $w_{3}=m_{2}(a_{2}a_{4}-a_{4}+1)$, and $w_{4}=m_{2}(a_{2}a_{3}-a_{2}+1)$. Then it is straightforward  to verify that $\tilde{w}_{2}+\tilde{w}_{3}+\tilde{w}_{4}=w_{2}+w_{3}+w_{4}$.
\end{proof}

Now, let us recall the notion of the Newton polyhedra:  write the monomial  $x_0^{a_0} \cdots x_n^{a_n}$ as $\mathbf{x}^\mathbf{a}$, where $\mathbf{a}=\left(a_0, \ldots, a_n\right) \in \mathbb{Z}^{n+1}.$ Given a polynomial function 
$\displaystyle{f(x)=\sum_{\mathbf{a} \in \mathbb{Z}_{+}^{n+1}} c_{\mathbf{a}} \mathbf{x}^\mathbf{a}}$, then the support of $f$ is defined as  
$$\operatorname{supp}(f):=\left\{\mathbf{a} \in \mathbb{Z}_{+}^{{n+1}} \mid c_{\mathbf{a}} \neq 0\right\}.$$ The 
Newton polyhedron $\Gamma_{+}(f)$ of  $f$ is defined to be the convex hull of the following set 
$${\bigcup_{\mathbf{a} \in \operatorname{supp}(f)}\left(\mathbf{a}+\mathbb{R}_{\geq 0}^{n+1}\right)}.$$ 
For each face $\gamma$ of $\Gamma_{+}(f)$, we define the polynomial $f_\gamma$ as follows:
$$
f_\gamma=\sum_{\mathbf{a} \in \gamma} c_{\mathbf{a}} \mathbf{x}^{\mathbf{a}}.
$$
A power series $f$ is said to be Newton non-degenerate, if for every face $\gamma$ the equation $f_\gamma=0$ defines a hypersurface smooth in the complement of the hypersurface $x_0 \cdots x_n=0$, that is, the intersection of the  singular set of the variety $V(f_\gamma),$ defined by $f_\gamma=0,$ and the open torus orbit,  defined by $x_0 \cdots x_n\not =0,$ is empty.
\medskip

Now we have the following criterion for canonical singularities (see \cite{IP} for instance).

\begin{lem}
    Let $S\subset\mathbb{C}^{n+1}$ be a normal hypersurface defined as the set of zeros of a Newton non-degenerate polynomial $f.$  If the point $(1,1,\dots,1)$ is in the interior of the Newton polyhedron $\Gamma_{+}(f)$, then $S$ has canonical singularities.
\end{lem}

Let $X$ be a hypersurface defined as in (2.2):$$ f=h\left(z_0, z_1\right)+z_4 z_2^{a_2}+z_2 z_3^{a_3}+z_3 z_4^{a_4}=0.$$ Since this weighted variety  is non-quasismooth, one can write  $h$ as either one of the following type of binomials:
\begin{enumerate}
\item[(A)] $ h(z_{0},z_{1})=z_{0}^{a_{0}}+z_{0}^{\beta_{0}}z_{1}^{\beta_{1}}.$
\item[(B)] $h(z_{0},z_{1})=z_{0}^{\alpha_{0}}z_{1}^{\alpha_{1}}+z_{0}^{\beta_{0}}z_{1}^{\beta_{1}}.$
\end{enumerate}
\medskip

We have the following proposition for the first kind of polynomials:

\begin{prop}
    Let ${\bf w}$ be a weight vector as described in (2.1) verifying $I<w_{0}+w_{1}$ and $X$ a non-quasismooth hypersurface of $\mathbb{P}({\bf w})$ defined by the equation $$X: h(z_{0},z_{1})+z_{4}z_{2}^{a_{2}}+z_{2}z_{3}^{a_{3}}+z_{3}z_{4}^{a_{4}}=0,$$ where $h(z_{0},z_{1})=z_{0}^{a_{0}}+z_{0}^{\beta_{0}}z_{1}^{\beta_{1}}$ is a binomial formed by elements of $\Lambda_{\textbf{w}}$. If $\beta_{0}<\dfrac{d}{w_{0}+w_{1}-I}$,
    then $X$ is klt.
\end{prop}
\begin{proof}
Since  $X$ is non-quasi-smooth,  $\beta_{0}>1.$  Also, as the number of monomials of the polynomial under study equals the number of variables, any linear combination of these monomials with non-zero coefficients defines a hypersurface isomorphic to $X.$  So we consider $X$ a general divisor in the linear system defined by the monomials of the defining equation. The base locus of this linear system is contained in the set of points
\begin{align*}
  B=\bigl\{ & [0:0:1:0:0], [0:0:0:1:0], [0:0:0:0:1], [0:1:1:0:0],[0:1:0:1:0],\\ & \hspace{4cm} [0:1:0:0:1]
      , [0:1:0:0:0]  \bigr\}.
\end{align*}
By  Bertini's theorem on $\mathbb{C}^{5}-\left\{0\right\}$, it follows that  $X$ is quasi-smooth outside of all these points. 
        Since the gradient of the polynomial $f$ defining $X$ is given by 
        {\small{
      \begin{itemize}
      \item   $\nabla f=\left(a_{0}z_{0}^{a_{0}-1}+\beta_{0}z_{0}^{\beta_{0}-1}z_{1}^{\beta_{1}},\beta_{1}z_{0}^{\beta_{0}}z_{1}^{\beta_{1}-1}, a_{2}z_{4}z_{2}^{a_{2}-1}+z_{3}^{a_{3}}, a_{3}z_{2}z_{3}^{a_{3}-1}+z_{4}^{a_{4}}, z_{2}^{a_{2}}+a_{4}z_{3}z_{4}^{a_{4}-1}\right),$ when $\beta_{1}>1.$
        \item $\nabla f=\left(a_{0}z_{0}^{a_{0}-1}+\beta_{0}z_{0}^{\beta_{0}-1}z_{1},z_{0}^{\beta_{0}}, a_{2}z_{4}z_{2}^{a_{2}-1}+z_{3}^{a_{3}}, a_{3}z_{2}z_{3}^{a_{3}-1}+z_{4}^{a_{4}}, z_{2}^{a_{2}}+a_{4}z_{3}z_{4}^{a_{4}-1}\right),$  when $\beta_{1}=1.$
\end{itemize}}} 
\noindent It follows that $X$ is quasismooth at all elements of $B$ except at the point $[0:1:0:0:0]$. Next, we will show that $X$ is klt in $[0:1:0:0:0]$. For this, we consider the hypersurface $\tilde{X}\subset\mathbb{P}(\textbf{w})$ defined by the general linear combination
        $$\tilde{X}: c_{0}z_{0}^{a_{0}}+c_{1}z_{0}^{\beta_{0}}z_{1}^{\beta_{1}}+c_{2}z_{4}z_{2}^{a_{2}}+c_{3}z_{2}z_{3}^{a_{3}}+c_{4}z_{3}z_{4}^{a_{4}}=0$$
        where $c_{i}$'s are complex numbers. As mentioned above, since the number of variables $z_{i}$'s is equal to the number of monomials, the varieties $X$ and $\tilde{X}$ are isomorphic. We will show that  $\tilde{X}$ is klt at the point  $[0:1:0:0:0]$.  In the affine chart $z_1 \neq 0$ we take $z_1=1$ and locally, around the singularity, $\tilde{X}$ is the quotient of the hypersurface      
$$\tilde{S}: c_{0}z_{0}^{a_{0}}+c_{1}z_{0}^{\beta_{0}}+c_{2}z_{4}z_{2}^{a_{2}}+c_{3}z_{2}z_{3}^{a_{3}}+c_{4}z_{3}z_{4}^{a_{4}}=0$$
in $\mathbb{C}^4$ by the group $\mathbb{Z}_{w_1}$. Since klt is a property preserved by finite quotients 
(see Corollary 2.43  in \cite{Ko3}), it suffices to show that such the general hypersurface $\tilde{S} \subset \mathbb{C}^4$ has canonical singularities. Clearly, $\tilde{S}$ is normal:  this hypersurface has a singular set of codimension at least 2. 
Also, since the coefficients of the monomials in this example are taken to be general, the Newton non-degeneracy is satisfied since the base locus of any collection of the monomials in the equation defining $\tilde{S}$  is contained in the hypersurface $x_0 \cdots x_n=0$.         
Thus according to Lemma 3.2,  the singularity in $\tilde{S}$ is canonical if the point $(1,1,1,1)$ is in the interior of the Newton polyhedron $\Gamma_{+}(g)$ generated  by the support of the polynomial $g$ defining $\tilde{S}$, that is, generated by the set 
        $$\operatorname{supp}(g)=\bigl\{ (a_{0},0,0,0), (\beta_{0},0,0,0),(0,a_{2},0,1), (0,1,a_{3},0), (0,0,1,a_{4})\bigr\}.$$
        We notice that it is enough to show that there exists some point in the Newton polyhedron  with all its entries less than 1. For this, we consider the following point
        \begin{equation*}
            P_{0}=\lambda_{1}(\beta_{0},0,0,0) + \lambda_{2}(0,a_{2},0,1)  + \lambda_{3}(0,1,a_{3},0)+\lambda_{4}(0,0,1,a_{4}),
        \end{equation*}        
        where the numbers $\lambda_{i}$'s are given as follows: as $\beta_{0}<\frac{d}{w_{0}+w_{1}-I}$, we can choose a  $q\in\mathbb{Q}$ such that $\beta_{0}=\frac{d}{w_{0}+w_{1}-q}$. Then there exists $\epsilon>0$ such that $\epsilon\in ]\max\{q,I-1\},I[$. We define 
        $$\lambda_{1}=\dfrac{w_{0}+w_{1}-\epsilon}{d}\ 
  \ \ \ \mbox{ and }\ \ \ \lambda_{i}=\frac{\tilde{w}_{i}-\delta}{d}, $$
        for $i=2,3,4$, where the $\tilde{w}_{i}$'s are defined in Lemma 3.1 and $\delta=\frac{I-\epsilon}{3}$. Clearly, each $\lambda_{i}$ is less than $1$. 
        Again by Lemma 3.1, we know that $\tilde{w}_{2}+\tilde{w}_{3}+\tilde{w}_{4}=w_{2}+w_{3}+w_{4}$. Then, we obtain
        $$\sum\lambda_{i}= \left( \dfrac{w_{0}+w_{1}-\epsilon}{d}\right) +\dfrac{\tilde{w}_{2}+\tilde{w}_{3}+\tilde{w}_{4}-3\delta}{d}=\dfrac{w_{0}+w_{1}+w_{2}+w_{3}+w_{4}-I}{d}=1$$
        Moreover, in the first coordinate we have  
        $$\lambda_{1}\beta_{0}=\left(\dfrac{w_{0}+w_{1}-\epsilon}{d}\right)\left(\dfrac{d}{w_{0}+w_{1}-q}\right)=\dfrac{w_{0}+w_{1}-\epsilon}{w_{0}+w_{1}-q}<1.$$
        Also, from Lemma 3.1, in the second coordinate we have:  
        $$\lambda_{2}a_{2}+\lambda_{3}=\dfrac{a_{2}\tilde{w}_{2}-\delta a_{2}+\tilde{w}_{3}-\delta}{d}=\dfrac{d-\delta(a_{2}+1)}{d}<1.$$
        Analogously, we can verify that the numbers $\lambda_{3}a_{3}+\lambda_{4}$ and $\lambda_{4}a_{4}+\lambda_{2}$ are less than $1$ in the remaining coordinates. 
        Hence the point $P_{0}$ is in $\Gamma_{+}(g)$. This implies that $\tilde{X}$ is klt in the point $[0:1:0:0:0]$. Hence, we have that  $X$ is klt in $[0:1:0:0:0]$. 
        
        \end{proof}
\medskip

There is a  similar result for the polynomials of the second kind, where  the binomial $h(z_{0},z_{1})$ is of the form (B). The proof of the next proposition  follows similar arguments to the ones used in Proposition 3.1 but we would rather include it for thoroughness.  
    \begin{prop}
        Let ${\bf w}$ be a weight vector as described  in (2.1) verifying $I<w_{1}+w_{0}$ and $X$ a non-quasi-smooth hypersurface of $\mathbb{P}({\bf w})$ defined by equation $$X: h(z_{0},z_{1})+z_{4}z_{2}^{a_{2}}+z_{2}z_{3}^{a_{3}}+z_{3}z_{4}^{a_{4}}=0,$$ where $h(z_{0},z_{1})=z_{0}^{{\alpha}_{0}}z_{1}^{\alpha_{1}}+z_{0}^{\beta_{0}}z_{1}^{\beta_{1}}$ is a binomial formed by elements of $\Lambda_{\textbf{w}}$ with $\alpha_{0}>\beta_{0}\geq 1$. If $\alpha_{1}<\dfrac{d}{w_{0}+w_{1}-I}$ and $\beta_{0}<\dfrac{d}{w_{0}+w_{1}-I}$,
    then $X$ is klt.
    \end{prop}
    \begin{proof}
We split the proof in two parts:
\medskip

\noindent $i)$ First, we discuss the case  $\alpha_{1}=1$. In this case, the hypersurface $X$ is given  by
        $$X: f=z_{0}^{\alpha_{0}}z_{1}+z_{0}^{\beta_{0}}z_{1}^{\beta_{1}}+z_{4}z_{2}^{a_{2}}+z_{2}z_{3}^{a_{3}}+z_{3}z_{4}^{a_{4}}=0.$$
        Since $X$ is non-quasi-smooth, we have $\beta_{1}>1$. 
         Also, as the number of monomials of $f$ equals the number of variables, any linear combination of these monomials with nonzero coefficients defines a hypersurface isomorphic to $X$. So we consider $X$ a general divisor in the linear system defined by the monomials of the defining equation. The base locus of this linear system is contained in the set of points
         {\small{
                \begin{align*}
            B=\bigl\{ & [0:0:1:0:0], [0:0:0:1:0], [0:0:0:0:1], [0:1:1:0:0],[0:1:0:1:0], [0:1:0:0:1]\\ & \hspace{2cm} [1:0:1:0:0], [1:0:0:1:0], [1:0:0:0:1], [1:0:0:0:0], [0:1:0:0:0] \bigr\}
        \end{align*}}}
        By Bertini's theorem on $\mathbb{C}^{5}-\{0\}$, we have that $X$ is quasi-smooth outside of all these points.  
        As $\beta_{1}>1$,  the gradient of $f$:
       {\small{$$\nabla f =\left(\alpha_{0}z_{0}^{\alpha_{0}-1}z_{1}+\beta_{0}z_{0}^{\beta_{0}-1}z_{1}^{\beta_{1}},z_{0}^{\alpha_{0}}+\beta_{1}z_{0}^{\beta_{0}}z_{1}^{\beta_{1}-1}, a_{2}z_{4}z_{2}^{a_{2}-1}+z_{3}^{a_{3}}, a_{3}z_{2}z_{3}^{a_{3}-1}+z_{4}^{a_{4}}, z_{2}^{a_{2}}+a_{4}z_{3}z_{4}^{a_{4}-1}\right).$$}}
 It follows that $X$ is quasi-smooth in all elements of $B$ except at the point $[0:1:0:0:0]$. Next, we  show that $X$ is klt in $[0:1:0:0:0]$.  In order to achieve this, we consider the hypersurface $\tilde{X} \subset \mathbb{P}(\mathbf{w})$ defined by the general linear combination 
        $$\tilde{X}:c_{0}z_{0}^{\alpha_{0}}z_{1}+c_{1}z_{0}^{\beta_{0}}z_{1}^{\beta_{1}}+c_{2}z_{4}z_{2}^{a_{2}}+c_{3}z_{2}z_{3}^{a_{3}}+c_{4}z_{3}z_{4}^{a_{4}}=0, $$
        where $c_{i}$'s are complex numbers. Since the number of variables $z_{i}$'s coincides with the number of monomials in $\tilde{X}$, the subvarieties $X$ and $\tilde{X}$ are isomorphic. Thus, it is enough to verify that $\tilde{X}$ is klt at the point $[0:1:0:0:0]$.   In the affine chart $z_1 \neq 0$ we take $z_1=1$ and locally, around the singularity, $\tilde{X}$ is the quotient of the hypersurface      
$$\tilde{S}:c_{0}z_{0}^{\alpha_{0}}+c_{1}z_{0}^{\beta_{0}}+c_{2}z_{4}z_{2}^{a_{2}}+c_{3}z_{2}z_{3}^{a_{3}}+c_{4}z_{3}z_{4}^{a_{4}}=0$$ in $\mathbb C^4$ by the group $\mathbb{Z}_{w_{1}}.$ As before,  it is enough  to show that such the general hypersurface $\tilde{S} \subset \mathbb{C}^4$ has canonical singularities. Clearly, $\tilde{S}$ is normal since  this hypersurface has  a singular set of codimension at least 2. 
Also, since the coefficients of the monomials in this example are taken to be general, the Newton non-degeneracy is satisfied since the base locus of any collection of the monomials in the equation  is contained in the hypersurface $x_0 \cdots x_n=0$.         
According to Lemma 3.2,  the singularity in $\tilde{S}$ is canonical if the point $(1,1,1,1)$ is in the interior of the Newton polyhedron $\Gamma_{+}(g)$ generated  by support of the polynomial $g$ defining $\tilde{S}$, that is, generated by the set 
 $$\operatorname{Supp}(g)=\bigl\{ (\alpha_{0},0,0,0), (\beta_{0},0,0,0),(0,a_{2},0,1), (0,1,a_{3},0), (0,0,1,a_{4})\bigr\}.$$
 Again, it suffices  to show that the Newton polyhedron contains some point $P_{0}$ with all its entries less than 1.
       Let us consider the point $P_{0}$:
        \begin{equation*}
            P_{0}=\lambda_{1}(\beta_{0},0,0,0) + \lambda_{2}(0,a_{2},0,1)  + \lambda_{3}(0,1,a_{3},0)+\lambda_{4}(0,0,1,a_{4})
        \end{equation*} 
        where the values $\lambda_{i}$'s are chosen to be the same as in Proposition 3.1. Since $\beta_{0}<\frac{d}{w_{0}+w_{1}-I}$, the same argument as the ones given before leads to the conclusion that $X$ is klt at $[0:1:0:0:0]$.
\medskip

\noindent $ii)$   Now, we will study the case when $\alpha_{1}>1$. In this case the hypersurface $X$ is defined by $$X: f=z_{0}^{\alpha_{0}}z_{1}^{\alpha_{1}}+z_{0}^{\beta_{0}}z_{1}^{\beta_{1}}+z_{4}z_{2}^{a_{2}}+z_{2}z_{3}^{a_{3}}+z_{3}z_{4}^{a_{4}}=0.$$
        If $\beta_{0}=1$, we arrive at the case described in $i)$, so  we assume $\beta_{0}>1$. 
As before we consider  $X$ a general divisor in the linear system defined by the monomials of the defining equation. The base locus of this linear system is contained in the set of points
{\small{
\begin{align*}
            B=\bigl\{ & [0:0:1:0:0], [0:0:0:1:0], [0:0:0:0:1], [0:1:1:0:0],[0:1:0:1:0], [0:1:0:0:1]\\ & \hspace{2cm} [1:0:1:0:0], [1:0:0:1:0], [1:0:0:0:1], [1:0:0:0:0], [0:1:0:0:0] \bigr\}.
\end{align*}
}}
By Bertini's theorem on $\mathbb{C}^{5}-\{0\}$, we have that $X$ is quasi-smooth outside of all points in $B$. From equating the gradient of $f$
        {\small{$$\left(\alpha_{0}z_{0}^{\alpha_{0}-1}z_{1}^{\alpha_{1}}+\beta_{0}z_{0}^{\beta_{0}-1}z_{1}^{\beta_{1}},\alpha_{1}z_{0}^{\alpha_{0}}z_{1}^{\alpha_{1}-1}+\beta_{1}z_{0}^{\beta_{0}}z_{1}^{\beta_{1}-1}, a_{2}z_{4}z_{2}^{a_{2}-1}+z_{3}^{a_{3}}, a_{3}z_{2}z_{3}^{a_{3}-1}+z_{4}^{a_{4}}, z_{2}^{a_{2}}+a_{4}z_{3}z_{4}^{a_{4}-1}\right)$$}}  to  zero, we conclude that $X$ is quasismooth at all elements of $B$ except  at the points 
        $[0:1:0:0:0]$ and $[1:0:0:0:0]$. 
        Next, we will show that $X$ is klt these two points. For this,  consider  the hypersurface $\tilde{X}\subset\mathbb{P}(\textbf{w})$, given  by the equation:
        $$\tilde{X}:c_{0}z_{0}^{\alpha_{0}}z_{1}^{\alpha_{1}}+c_{1}z_{0}^{\beta_{0}}z_{1}^{\beta_{1}}+c_{2}z_{4}z_{2}^{a_{2}}+c_{3}z_{2}z_{3}^{a_{3}}+c_{4}z_{3}z_{4}^{a_{4}}=0, $$
        where $c_{i}$'s are complex numbers. As before, the subvarieties $X$ and $\tilde{X}$ are isomorphic, so  we will verify that  $\tilde{X}$ is klt at the points $[0:1:0:0:0]$ and $[1:0:0:0:0].$ 
        
For  the point $[0:1:0:0:0], $  in the affine chart $z_1 \neq 0$ we take $z_1=1$ and locally, around the singularity, $\tilde{X}$ is the quotient of the hypersurface    
$$\tilde{S}:c_{0}z_{0}^{\alpha_{0}}+c_{1}z_{0}^{\beta_{0}}+c_{2}z_{4}z_{2}^{a_{2}}+c_{3}z_{2}z_{3}^{a_{3}}+c_{4}z_{3}z_{4}^{a_{4}}=0$$  in $\mathbb C^4$ by the group $\mathbb{Z}_{w_{1}}.$ The normality of the hypersurface and the Newton non-degeneracy of the polynomial describing the hypersurface  are established using similar arguments as the ones used before.  Thus, it is enough  to show that such the general hypersurface $\tilde{S} \subset \mathbb{C}^4$ has canonical singularities.
 By  Lemma 3.2, the singularity in $S$ is canonical if the point $(1,1,1,1)$ is in the interior of the Newton polyhedron $\Gamma_{+}(g)$ generated by support of the polynomial $g$ defining $\tilde{S}$, that is, generated by the set $$\Gamma_{+}(g)=\bigl\{ (\alpha_{0},0,0,0), (\beta_{0},0,0,0),(0,a_{2},0,1), (0,1,a_{3},0), (0,0,1,a_{4})\bigr\}.$$
As before,  we show that there exists some point $P_{0}$ with all its entries less than $1$ lying on $\Gamma_{+}(g)$. Consider the point $P_{0}$:
        \begin{equation}
            P_{0}=\lambda_{1}(\beta_{0},0,0,0) + \lambda_{2}(0,a_{2},0,1)  + \lambda_{3}(0,1,a_{3},0)+\lambda_{4}(0,0,1,a_{4})
        \end{equation} 
        where the values $\lambda_{i}$'s are choose as in Proposition 3.1. Again, as $\beta_{0}<\frac{d}{w_{0}+w_{1}-I}$, by a similar way as above, we conclude $X$ is klt in the point $[0:1:0:0:0]$. For  the point $[1:0:0:0:0]$ one uses identical steps as the ones presented above.    

    \end{proof}


\section{Exceptionality of  klt Fano 3-folds}


In the previous section, we have shown the existence of families of non-quasi-smooth hypersurfaces whose singularities are klt. More precisely, these hypersurfaces $X\subset \mathbb{P}(\textbf{w})$ are described by polynomials of the following forms:
\begin{align*}
    \mbox{ Type A} :&  f=z_{0}^{a_{0}}+z_{0}^{\beta_{0}}z_{1}^{\beta_{1}}+z_{4}z_{2}^{a_{2}}+z_{2}z_{3}^{a_{3}}+z_{3}z_{4}^{a_{4}}=0, \ \mbox{ where }\beta_{0}\geq 2, \\
    \mbox{ Type B} : &  f=z_{0}^{\alpha_{0}}z_{1}^{\alpha_{1}}+z_{0}^{\beta_{0}}z_{1}^{\beta_{1}}+z_{4}z_{2}^{a_{2}}+z_{2}z_{3}^{a_{3}}+z_{3}z_{4}^{a_{4}}=0, \ \mbox{ where }\beta_{0}<\alpha_{0},
\end{align*}
 where the weight vector ${\bf w}$ has the form
 \begin{equation}
     {\bf w}=(w_{0},w_{1},w_{2},w_{3},w_{4})=(m_{3}v_{0},m_{3}v_{1},m_{2}v_{2},m_{2}v_{3},m_{2}v_{4})
 \end{equation}
 and conditions (2.1) are satisfied. It is not difficult to show that weight vectors as above determine quasismooth hypersurfaces in $\mathbb{P}(\textbf{w})$ which consist of the zero set of polynomials of type BP-cycle, chain-cycle or cycle-cycle. Moreover, from Theorem 2.1,   for  weights verifying  the inequality $$Id < \dfrac{4}{3}\min_{i,j}\{w_{i}w_{j}\}$$ the corresponding  quasismooth hypersurfaces admit  Kähler-Einstein metrics.  Naturally, one asks whether  it is possible that the associated non-quasi-smooth hypersurface $X$ as above, which can be considered as an element in the boundary of the  closure of the moduli of quasismooth hypersurfaces, admits a Kähler-Einstein singular  metric. In fact, we will prove that some members of these families of klt Fano varieties are exceptional and hence they admit Kähler-Einstein singular  metrics. 

From Lemma 2.3,  it is enough to show that $\mbox{glct}(X)>1$, thus it suffices to show  that there exists some $\lambda>1$ such that the pair $(X,\lambda D)$ is lc for any $D\in |-K_{X}|$. We will use the multiplicity of $X$ at each point $p$ to show this.  
\medskip

For smooth points, we have the following result given in  \cite{Ko2}, Claim 2.10.4:
\begin{lem}
    Let $X$ be a regular scheme and $D$ an effective $\mathbb{Q}$-divisor. Then the pair $(X,D)$ is canonical if $mult_{p} D\leq 1$ for every point $p\in X$.
\end{lem}

Recall that  the multiplicity at a point of an irreducible closed substack or an effective algebraic cycle in weighted projective space $\mathbb P(\bf w)$ is defined  to be the multiplicity at a corresponding point of its inverse image in any orbifold chart $\mathbb{C}^n \rightarrow\left[\mathbb{C}^n / \mu_{a_i}\right] \cong\left\{x_i \neq 0\right\} \subset \mathbb P(\bf w)$. This definition is independent of the index $i$, since the different orbifold charts are étale-locally isomorphic.  We have  the following lemma that determines  bounds for the multiplicity in this setting (\cite{JK1}  Proposition 11).

\begin{lem} Let $Y \subset \mathbb{P}(w_{0},w_{1},\dots, w_{n})$ be a $m$-dimensional subvariety of a weighted projective space. Assume that $Y$ is not contained in the singular locus and that $w_{0}\geq w_{1}\geq \dots \geq w_{n}$. Let $Y_i \subset \mathbb{A}^n$ denote the preimage of $Y$ in the orbifold chart

$$
\mathbb{A}^n \rightarrow \mathbb{A}^n / \mathbb{Z}_{a_i} \cong \mathbb{P}\left(a_0, \ldots, a_n\right) \backslash\left(x_i=0\right).
$$
Then for every $i$ and every $p \in Y_i$, 
$$
\operatorname{mult}_p Y_i \leq (w_{0}w_{1}\dots w_{m})\left(Y \cdot \mathcal{O}(1)^m\right),
$$ 
That is, $$mult_{p}Y\leq (w_{0}w_{1}\dots w_{m})deg(Y)$$ 
    for each $p$ in $Y$.
\end{lem}

Using the tools given above, we can show  the exceptionality of the quasismooth points. On the other hand, at the points where $X$ is not quasismooth, we use a rather new technique developed by Totaro in \cite{To} which involves the notion of weighted tangent cone which considers the insights of a  new method of resolution in charactertistic 0 which is much better suited for computations of various birational invariants, such as $\log$ canonical thresholds: stack-theoretic weighted blow-ups  \cite{ATW, Mc}, 
For explicit coordinate charts of the stack-theoretic weighted blow-up see Section 3.4 in \cite{ATW}. 

\begin{defi}
    Let $X\subset\mathbb{C}^{n+1}$ be a hypersurface that contains the origin and let ${\bf w}=(w_{0},w_{1},\dots w_{n})$ be a weight vector. Given the stack-theoretic weighted blow-up $f:X'\to X$ at the origin with respect to $\bf w$, we define its $\bf w$-weighted tangent cone $X^c\subset\mathbb{P}({\bf w})$ as the exceptional divisor in $X'$. Here  $\mathbb{P}({\bf w})=\left[\left(\mathbb{A}^{n+1}-0\right) / \mathbb{C}^*\right]$ is viewed as a stack.
\end{defi}

In case there is no ambiguity, we suppress the prefix $\bf w$ and write $X^c$ as the weighted tangent cone of the hypersurface $X$. In the following lemma, whose proof can be found in \cite{To}, we see that the weighted blow-up $f:X'\to X$ preserves a singularity of type lc under certain conditions.   
\begin{lem}
    Let $X$ be a hypersurface in $\mathbb{C}^{n+1}$ that contains the origin, and $D$ be an effective $\mathbb{Q}$-Cartier $\mathbb{Q}$-divisor in $X$.  If $D^{c}$ denotes the weighted tangent cone of $D$, then $K_{X^{c}}+D^{c}\sim_{\mathbb{Q}}\mathcal{O}_{X^{c}}(r)$ for some $r\in\mathbb{Q}$. Moreover, if the pair $(X^{c},D^{c})$ is lc and $r\leq0$, then $(X,D)$ is lc near the origin.
\end{lem}

We have the following definition.
\begin{defi}
    Let $S\subset A^{n+1}$ be a closed subscheme. Given a weight vector ${\bf w}=(w_{0},w_{1},\dots, w_{n})$, we define the weighted multiplicity of $S$ at the origin as
    $$mult_{{\bf w}}S = deg(S^c)$$
    where $S^c$ is the weighted tangent cone of $S$  seen as a substack of $\mathbb{P}({\bf w})$.
\end{defi}
Here, Totaro obtains a bound for the weighted multiplicity (see Lemma 4.4 in \cite{To}).
\begin{lem}
    Let $X\subset\mathbb{P}({\bf{w}})$ be an irreducible closed substack, where ${\bf w}=(w_{0},w_{1},\dots, w_{n+1})$. In $z_{n+1}=1$, we see $X$ as a variety $\tilde{X}\subset\mathbb{C}^{n+1}$. Then the weighted  multiplicity of $\tilde{X}$ at the origin with respect to the weight vector ${\tilde{\bf w}}=(w_{0},w_{1},\dots,w_{n})$ verifies the inequality
    $$mult_{\tilde{\bf w}}(\tilde{X})\leq w_{n+1}deg(X)$$
\end{lem}

Now, we will prove the following result for hypersurfaces $X$ defined by polynomials of type A. 
 \begin{prop}
     Let $X: f=0$ be a Fano hypersurface in $\mathbb{P}({\bf w})$ where the weight vector ${\bf w}$ is given as in 2.1 subject to $w_{0}+w_{1}>I$ and $f$ is a polynomial of type A, that is, $$f=z_{0}^{a_{0}}+z_{0}^{\beta_{0}}z_{1}^{\beta_{1}}+z_{4}z_{2}^{a_{2}}+z_{2}z_{3}^{a_{3}}+z_{3}z_{4}^{a_{4}}=0, \ \mbox{ where }\beta_{0}\geq 2$$  of degree $d=m_{2}m_{3}$, such that $\beta_{0}\mid m_{3}$ and $\beta_{0}\in\left]\frac{I}{v_{0}m_{2}^2}; \frac{d}{w_{0}+w_{1}-I}\right[$. If $Id<\min\{w_{i}w_{j}\}$, then $X$ is exceptional.
 \end{prop}
\begin{proof}
    Let $X\subset\mathbb{P}(\textbf{w})$ be a Fano hypersurface defined by the polynomial $$X:f=z_{0}^{a_{0}}+z_{0}^{\beta_{0}}z_{1}^{\beta_{1}}+z_{4}z_{2}^{a_{2}}+z_{2}z_{3}^{a_{3}}+z_{3}z_{4}^{a_{4}}=0$$
    Since $\beta_{0}<\frac{d}{w_{0}+w_{1}-I}$, we obtain by a previous proposition that $X$ is klt. To show that $X$ is exceptional, by Lemma 2.4, we must prove that $\mbox{glct}(X)>1$. This means that we should ensure that there exists some $\lambda >1$ such that $(X,\lambda D)$ is lc for any $\mathbb{Q}$-divisor $D\in |-K_{X}|$. 
    
    We take $D\in|-K_{X}|$ to be a $\mathbb{Q}$-divisor on $X$. Then for any smooth point $p\in X$, we will show that there exists a real number $\lambda>1$ that does not depend of $D$ or the point $p$, such that $mult_{p}(\lambda D)\leq 1$. We notice that this is equivalent to
    $$mult_{p}(\lambda D)\leq 1 \Leftrightarrow \lambda \leq\frac{1}{mult_{p}D}$$
    As $D\sim_{\mathbb{Q}} -K_{X}=O_{X}(I)$  by the adjunction formula and since  $\dim(D)=2$, we obtain by Lemma 4.2:
$$ mult_{p}D\leq \max\{w_{i}w_{j}w_{k}\}deg(D)=\max\{w_{i}w_{j}w_{k}\}\left(-K_{X}\right) \cdot c_1(O(1))^{2},$$
 which equals  $$\dfrac{\max\{w_{i}w_{j}w_{k}\}Id}{w_{0}w_{1}w_{2}w_{3}w_{4}} = \dfrac{Id}{\min\{w_{i}w_{j}\}}<1,$$
where the intersection numbers are computed on the $3$-dimensional stack $X \subset \mathbb{P}(w_0, w_1, w_2, w_3, w_4).$ This implies that 
    $$\frac{1}{mult_{p}D}\geq \frac{\min\{ w_{i}w_{j}\}}{Id}>1.$$
    Then, we can take $\lambda\in \left] 1,\frac{\min\{ w_{i}w_{j}\}}{Id}\right[$ such that $mult_{p}(\lambda D)\leq 1$ for any $D\in|-K_{X}|$ and smooth point $p\in X$. By Lemma 4.1, we conclude that the pair $(X,\lambda D)$ is lc outside of the point $[0:1:0:0:0]$ when $\lambda\in \left] 1,\frac{\min\{ w_{i}w_{j}\}}{Id}\right[$.

    Now, we want to prove that $X$ is exceptional in the non-quasi-smooth point $[0:1:0:0:0]$. In the hyperplane $z_{1}=1$, the hypersurface $X$ is defined as 
    $$X:z_{0}^{a_{0}}+z_{0}^{\beta_{0}}+z_{4}z_{2}^{a_{2}}+z_{2}z_{3}^{a_{3}}+z_{3}z_{4}^{a_{4}}=0.$$
    Notice that this equation also defines a hypersurface $\tilde{X}$ in the affine space $\mathbb{C}^{4}$. Let us see that there exists some $\lambda>1$ such that the pair $(\tilde{X},\lambda \tilde{D})$ is lc near the origin, where $\tilde{D}$ is the affine version of the divisor $D$ restricted to $z_{1}=1$ . Here, we establish the weighted tangent cone $\tilde{X}^c\subset \mathbb{P}(\textbf{w}^c)$, to the hypersurface $\tilde{X}$ in the origin with respect to $\textbf{w}^{c}=(r,v_{2},v_{3},v_{4})$, where $r=m_{3}/\beta_{0}$. This is defined by the polynomial of degree $m_{3}$: 
    $$\tilde{X}^{c}:z_{0}^{\beta_{0}}+z_{4}z_{2}^{a_{2}}+z_{2}z_{3}^{a_{3}}+z_{3}z_{4}^{a_{4}}=0.$$ The monomial $z_{0}^{a_{0}}$ is not taken into account because its degree $d$ is greater than $m_{3}$. Clearly, $\tilde{X}^c$ is a quasi-smooth hypersurface. Moreover, since $\textbf{w}=(w_{0},w_{1},w_{2},w_{3},w_{4})$ is well-formed and $\gcd(v_{i},v_{j})=1$, we obtain that the hypersurface  $\tilde{X}^{c}\subset \mathbb{P}(\textbf{w}^c)$ is well-formed. 
    
    On the other hand, as $X$ is klt and $\tilde{X}^{c}$ is obtained by taking the quotient of $X$ restricted to $z_{1}=1$ by a finite group, we have $\tilde{X}^c$ is klt. In addition, if we write $\tilde{D}^{c}$ as the weighted tangent cone of $\tilde{D}$, which is $D$ restricted to $z_{1}=1$, we get that $\tilde{D}^{c}$ is an effective $\mathbb{Q}$-Cartier $\mathbb{Q}$-divisor. By Lemma 4.3, we have that $K_{\tilde{X}^c}+\lambda \tilde{D}^{c}\sim_{\mathbb{Q}}\mathcal{O}_{\tilde{X}^{c}}(q)$ for some $q\in\mathbb{Q}$, where $\lambda\in\mathbb{R}$. Next, in order to apply the second result of the Lemma 4.3, we shall show that there exists some $\lambda >1$ such that $q\leq 0$ and the pair $(\tilde{X}^c,\lambda \tilde{D}^c)$ is lc.
    \begin{itemize}
        \item \textbf{Claim 1:} There exists some $\lambda>1$, such that $q\leq 0$. 

        First, we will prove that the index $I^c=r+v_{2}+v_{3}+v_{4}-m_{3}$ of $\textbf{w}^c$ is greater than $1$. Since the hypersurface $\tilde{X}^c$ is well-formed, we obtain by the Adjuntion formula
        $$-K_{\tilde{X}^c}=\mathcal{O}_{\tilde{X}^c}(r+v_{2}+v_{3}+v_{4}-m_{3}).$$
        Due to the index $I=\sum w_{i}-d>0$ and $d=m_{2}m_{3}$, then we have 
        \begin{equation}
            m_{2}(v_{2}+v_{3}+v_{4})=d-w_{0}-w_{1}+I \Leftrightarrow v_{2}+v_{3}+v_{4}=m_{3}\left(1-\dfrac{w_{0}+w_{1}-I}{d}\right).
        \end{equation}
        Moreover, as $m_{3}=\beta_{0}r$ and $\beta_{0}<\dfrac{d}{w_{0}+w_{1}-I}$, we obtain
        \begin{equation*}
        m_{3}\left(1-\dfrac{w_{0}+w_{1}-I}{d}\right)>m_{3}\left(1-\dfrac{1}{\beta_{0}}\right)  = m_{3}-r.
        \end{equation*}
        Putting this into of the equality of the right-side in (4.2), we had proved $r+v_{2}+v_{3}+v_{4}-m_{3}>0$. Thus, we get $I^{c}\geq 1$. On the other hand, seeing that $K_{\tilde{X}^c}+\lambda \tilde{D}^{c}\sim_{\mathbb{Q}}\mathcal{O}_{\tilde{X}^{c}}(q)$, we have
        $$q\leq 0  \Leftrightarrow  \lambda \deg(\tilde{D}^c)\leq \deg(-K_{X^c}) = I^c\left(\dfrac{m_{3}}{rv_{2}v_{3}v_{4}}\right).$$
        Then
        \begin{equation}
            q\leq 0 \Leftrightarrow \lambda \leq \dfrac{I^cm_{3}}{rv_{2}v_{3}v_{4}\deg(\tilde{D}^c)}.
        \end{equation}
        In addition, by Lemma 4.4, we obtain
        \begin{equation}
            \deg(\tilde{D}^c)=mult_{\tilde{\bf w}^c}(\tilde{D})\leq w_{1}\deg(D) = w_{1}\left(\dfrac{Id}{w_{0}w_{1}w_{2}w_{3}w_{4}}\right)=\dfrac{Id}{w_{0}w_{2}w_{3}w_{4}}.
        \end{equation}
        This implies that
        $$\dfrac{I^cm_{3}}{rv_{2}v_{3}v_{4}\deg(\tilde{D}^c)}\geq \left(\dfrac{I^cm_{3}}{rv_{2}v_{3}v_{4}}\right)\left(\dfrac{w_{0}w_{2}w_{3}w_{4}}{Id}\right) = \dfrac{I^cw_{0}m_{2}^2}{Ir}.$$
        As $w_{0}=r\beta_{0}v_{0}$ and $m_{2}>1$, then we have
        $$\dfrac{I^cm_{3}}{rv_{2}v_{3}v_{4}\deg(\tilde{D}^c)}\geq\dfrac{I^cw_{0}m_{2}^2}{Ir}=\dfrac{I^c\beta_{0}v_{0}m_{2}^2}{I}>1.$$
        Therefore, we can choose some $\lambda\in\left]1,\dfrac{I^cw_{0}m_{2}^2}{Ir}\right[$, which implies that $q\leq 0$.
        \item \textbf{Claim 2:} There exists some $\lambda>1$, such that the pair $(\tilde{X}^c,\lambda \tilde{D}^c)$ is lc.

        In order to apply the Lemma 4.1, we note that $\tilde{X}^c$ is quasi-smooth and $\tilde{D}^c$ is an effective $\mathbb{Q}$-divisor. Thus, we just need to prove that there exists some $\lambda>1$ such that $mult_{p}(\lambda \tilde{D}^c)\leq 1$, for any $p\in \tilde{X}^c$. Indeed,  we have 
        \begin{equation}
            mult_{p}(\lambda \tilde{D}^c)\leq 1 \Leftrightarrow \lambda \leq \dfrac{1}{mult_{p}(\tilde{D}^c)}.
        \end{equation}
        On the other hand, since $\dim(\tilde{D}^c)=1$, by Lemma 4.2 we obtain
        \begin{equation}
           mult_{p}(\tilde{D}^c)\leq \max\{w_{i}^{c}w_{j}^{c}\}\deg(\tilde{D}^{c}) \Rightarrow  \dfrac{1}{mult_{p}(\tilde{D}^c)}\geq \dfrac{1}{ \max\{w_{i}^{c}w_{j}^{c}\}\deg(\tilde{D}^{c})},
        \end{equation}
        where $w_{i}^{c}$ and $w_{j}^{c}$ are weights in $\textbf{w}^c=(r,v_{2},v_{3},v_{4})$.
        
        Here, we have two possible scenarios:
        \begin{itemize}
            \item If $\max\{w_{i}^{c}w_{j}^{c}\}=rv_{i}$, where $i\in\{2,3,4\}$.

            In this case, we can suppose without loss of generality that $\max\{w_{i}^{c}w_{j}^{c}\}=rv_{2}$. From the inequality in (4.6) and (4.4), we get
            $$\dfrac{1}{mult_{p}(\tilde{D}^c)}\geq \dfrac{1}{rv_{2}\deg(\tilde{D}^c)}\geq \dfrac{1}{rv_{2}}\left(\dfrac{w_{0}w_{2}w_{3}w_{4}}{Id}\right).$$
            As $Id<w_{3}w_{4}$ and $rv_{2}<w_{0}w_{2}$, we have
            $$\dfrac{1}{mult_{p}(\tilde{D}^c)}\geq \dfrac{1}{rv_{2}}\left(\dfrac{w_{0}w_{2}w_{3}w_{4}}{Id}\right)>1.$$
            Then,  we can choose some $\lambda\in\left]1,\dfrac{1}{rv_{2}}\left(\dfrac{w_{0}w_{2}w_{3}w_{4}}{Id}\right)\right[$ such that it satisfies (4.5).
            \item If $\max\{w_{i}^{c}w_{j}^{c}\}=v_{i}v_{j}$, where $i,j\in\{2,3,4\}$.
            Here we can assume without loss of generality that $\max\{w_{i}^{c}w_{j}^{c}\}=v_{2}v_{3}$. Again, using the inequality in (4.6) and (4.4), we obtain
            $$\dfrac{1}{mult_{p}(\tilde{D}^c)}\geq \dfrac{1}{v_{2}v_{3}}\left(\dfrac{w_{0}w_{2}w_{3}w_{4}}{Id}\right).$$
            As $Id<w_{0}w_{4}$ and $v_{2}v_{3}<w_{2}w_{3}$, we have
            $$\dfrac{1}{mult_{p}(\tilde{D}^c)}\geq\dfrac{1}{v_{2}v_{3}}\left(\dfrac{w_{0}w_{2}w_{3}w_{4}}{Id}\right)>1.$$
            In a similar way to above, we can choose some $\lambda\in\left]1,\dfrac{1}{v_{2}v_{3}}\left(\dfrac{w_{0}w_{2}w_{3}w_{4}}{Id}\right)\right[$.  
        \end{itemize}
        
    \end{itemize}
    By the claims given above, there exists some $\lambda>1$ such that $q\leq 0$ and the pair $(\tilde{X}^c,\lambda \tilde{D}^c)$ is canonical, which in turn implies that $(\tilde{X}^c,\lambda \tilde{D}^c)$ is lc. Thus, from Lemma 4.3, we have that the pair $(\tilde{X},\lambda \tilde{D})$ is lc. Finally, if we write
        $$\epsilon =\min\left\{\frac{\min\{ w_{i}w_{j}\}}{Id}, \dfrac{I^cw_{0}m_{2}^2}{Ir},\dfrac{1}{rv_{2}}\left(\dfrac{w_{0}w_{2}w_{3}w_{4}}{Id}\right), \dfrac{1}{v_{2}v_{3}}\left(\dfrac{w_{0}w_{2}w_{3}w_{4}}{Id}\right)\right\}>1,$$
        we have that when $\lambda\in]1,\epsilon[$, the pair $(X,\lambda D)$ is lc for any $\mathbb{Q}$-divisor $D\in |-K_{X}|$. By Lemma 2.4, it implies that $X$ is exceptional.
 
\end{proof}

Now, we study the case of polynomials of type B. Here we have the following proposition.
\begin{prop}
     Let $X: f=0$ be a Fano hypersurface in $\mathbb{P}({\bf w})$ where the weight vector ${\bf w}$ is given as in (2.1) subject to $w_{0}+w_{1}>I$ and $f$ is a polynomial of type B, that is, $$f=z_{0}^{\alpha_{0}}z_{1}^{\alpha_{1}}+z_{0}^{\beta_{0}}z_{1}^{\beta_{1}}+z_{4}z_{2}^{a_{2}}+z_{2}z_{3}^{a_{3}}+z_{3}z_{4}^{a_{4}}=0, \ \mbox{ where }\beta_{0}<\alpha_{0},$$  of degree $d=m_{2}m_{3}$, such that $\alpha_{1}$ and $\beta_{0}$ divide to $ m_{3}$, $\beta_{0}\in\left] \frac{I}{v_{0}m_{2}^2},\frac{d}{w_{0}+w_{1}-I}\right[$ and $\alpha_{1}\in\left] \frac{I}{v_{1}m_{2}^2},\frac{d}{w_{0}+w_{1}-I}\right[$. If $Id<\min\{w_{i}w_{j}\}$, then $X$ is exceptional.
 \end{prop}
 \begin{proof}
     Let $X\subset\mathbb{P}(\textbf{w})$ be Fano hypersurface defined by the polynomial 
     $$X: f=z_{0}^{\alpha_{0}}z_{1}^{\alpha_{1}}+z_{0}^{\beta_{0}}z_{1}^{\beta_{1}}+z_{4}z_{2}^{a_{2}}+z_{2}z_{3}^{a_{3}}+z_{3}z_{4}^{a_{4}}=0.$$
     From Proposition 3.2 it follows that $X$ is klt.  Now we consider two cases: $\alpha_{1}=1$ or $\alpha_{1}\geq2$.
     \begin{itemize}
         \item First, if  $\alpha_{1}=1$. Then $X$ is defined by
         $$X: f=z_{0}^{\alpha_{0}}z_{1}+z_{0}^{\beta_{0}}z_{1}^{\beta_{1}}+z_{4}z_{2}^{a_{2}}+z_{2}z_{3}^{a_{3}}+z_{3}z_{4}^{a_{4}}=0.$$
         As $\beta_{1}\geq 2$, then the terms $z_{0}^{\alpha_{0}}z_{1}$ and $z_{0}^{\beta_{0}}z_{1}^{\beta_{1}}$ are different, where $\alpha_{0}>\beta_{0}$.
         In addition, we see that $X$ is quasi-smooth in all points except $[0:1:0:0:0]$. Since $Id<\min\{w_{i}w_{j}\}$, we can use a similar argument as in Proposition 4.1 and we can find some rational number $\lambda_{1}>1$ such that the pair $(X,\lambda D)$ is lc outside of the point $[0:1:0:0:0]$ for any $D\in | -K_{X}|$. On the other hand, to prove that $X$ is exceptional in the non-quasi-smooth point $[0:1:0:0:0]$, we will work with the weighted tangent cone on the intersection of $X$ with the hyperplane $z_{1}=1$. More precisely, we have that $X$ is described in $z_{1}=1$ as follows
         $$X: z_{0}^{\alpha_{0}}+z_{0}^{\beta_{0}}+z_{4}z_{2}^{a_{2}}+z_{2}z_{3}^{a_{3}}+z_{3}z_{4}^{a_{4}}=0.$$
         Clearly, this equation defines a hypersurface $\tilde{X}$ in the affine space $\mathbb{C}^4$. Similarly to Proposition 4.1, we obtain the divisor $\tilde{D}$ from $D$.  Next, we will prove that there exists some $\lambda>1$ such that the pair $(\tilde{X},\lambda \tilde{D})$ is lc near the origin. For this, we consider the weighted tangent cone $\tilde{X}^c\subset\mathbb{P}(\textbf{w}^{c})$ to the hypersurface $\tilde{X}\subset\mathbb{C}^{4}$ in the origin, which is defined by the polynomial of degree $m_{3}$
          $$\tilde{X}^{c}:  z_{0}^{\beta_{0}}+z_{4}z_{2}^{a_{2}}+z_{2}z_{3}^{a_{3}}+z_{3}z_{4}^{a_{4}}=0,$$
           where $\textbf{w}^{c}=(r,v_{2},v_{3},v_{4})$ and $r=m_{3}/\beta_{0}$. Here, we suppress the term $z_{0}^{\alpha_{0}}$ because its degree is greater than $m_{3}$. We notice that the polynomial defining $\tilde{X}$ and the weight vector $\textbf{w}^c$ are similar as in Proposition 4.1. Moreover, as $\tilde{X}^c$ is well-formed, $Id<\min\{w_{i}w_{j}\}$ and $\beta_{0}\in\left] \frac{I}{v_{0}m_{2}^2},\frac{d}{w_{0}+w_{1}-I}\right[$, then we can repeat the argument used in Proposition 4.1. Thus, we find some $\lambda_{2}>1$ such that the pair $(\tilde{X}^c,\lambda_{2}\tilde{D}^c)$ is lc and $K_{\tilde{X}^c}+\lambda_{2} \tilde{D}^{c}\sim_{\mathbb{Q}}\mathcal{O}_{\tilde{X}^{c}}(q)$, with $q\leq0$,  where $\tilde{D}^c$ is the weighted tangent cone of $\tilde{D}$ with respect to weight vector $\textbf{w}^c$.  Lemma 4.3  implies that $(\tilde{X},\lambda_{2}\tilde{D})$ is lc near the origin. So, if we take $\lambda=\min\{\lambda_{1},\lambda_{2}\}>1$, we have that the pair $(X,\lambda D)$ is lc for any $D\in |-K_{X}|$. Finally, by Lemma 2.4, we conclude that $X$ is exceptional.
           
           \item When $\alpha_{1}\geq 2$, the hypersurface $X$ is defined by
           $$X:f=z_{0}^{\alpha_{0}}z_{1}^{\alpha_{1}}+z_{0}^{\beta_{0}}z_{1}^{\beta_{1}}+z_{4}z_{2}^{a_{2}}+z_{2}z_{3}^{a_{3}}+z_{3}z_{4}^{a_{4}}=0.$$
           Clearly, this is quasi-smooth in all points except at the points $[1:0:0:0:0]$ and $[0:1:0:0:0]$. As $Id<\min\{w_{i}w_{j}\}$, we repeat the same process given above and obtain a value $\lambda_{1}>1$ such that the pair $(X,\lambda_{1} D)$ is lc in the smooth part of $X$. On the other hand, to verify that $X$ is exceptional at the points $[1:0:0:0:0]$ and $[0:1:0:0:0]$, we will use again the notion of  weighted tangent cones.
           
           First, we intersect $X$ with the hyperplane $z_{0}=1$. Here, we describe the hypersurface $X$ as $$X:f=z_{1}^{\alpha_{1}}+z_{1}^{\beta_{1}}+z_{4}z_{2}^{a_{2}}+z_{2}z_{3}^{a_{3}}+z_{3}z_{4}^{a_{4}}=0.$$
           Notice that this equation defines a hypersurface $\tilde{X}$ in the affine space $\mathbb{C}^4$. Moreover, we denote by $\tilde{D}$  the divisor coming from $D$ in $z_{0}=1$. Next, we  prove that there exists some $\lambda>1$ such that the pair $(\tilde{X},\tilde{D})$ is lc near the origin. 
           Let $\tilde{X}^c\subset\mathbb{P}(\textbf{w}^c)$ be the weighted tangent cone to the hypersurface $\tilde{X}$ in the origin, which is defined by the following polynomial of degree $m_{3}$
           $$\tilde{X}^{c}: z_{1}^{\alpha_{1}}+z_{4}z_{2}^{a_{2}}+z_{2}z_{3}^{a_{3}}+z_{3}z_{4}^{a_{4}}=0, $$
           where $\textbf{w}^{c}=(r,v_{2},v_{3},v_{4})$ and $r=m_{3}/\alpha_{1}$. Since the degree of term $z_{1}^{\beta_{1}}$ is greater than $m_{3}$, we do not take it  into account. Now, we notice that the polynomial defining $\tilde{X}$ and the weight vector $\textbf{w}^c$ are similar as in Proposition 4.1. Moreover, as the hypersurface 
           $\tilde{X}^{c}$ is well-formed, $d<\min\{w_{i}w_{j}\}$ and $\alpha_{1}\in\left] \frac{I}{v_{1}m_{2}^2},\frac{d}{w_{0}+w_{1}-I}\right[$, then once again, we use the   argument on the proof of  Proposition 4.1. This means that we can find some $\lambda_{2}>1$ such that the pair $(\tilde{X}^c,\lambda_{2}\tilde{D}^c)$ is lc and $K_{\tilde{X}^c}+\lambda_{2} \tilde{D}^{c}\sim_{\mathbb{Q}}\mathcal{O}_{\tilde{X}^{c}}(q)$, with $q\leq0$,  where $\tilde{D}^c$ is the weighted tangent cone of $\tilde{D}$ with respect to weight vector $\textbf{w}^c$. By Lemma 4.3, it implies that $(\tilde{X},\lambda_{2}\tilde{D})$ is lc near the origin.

           On the other hand, since  $Id<\min\{w_{i}w_{j}\}$ and $\beta_{0}\in\left] \frac{I}{v_{0}m_{2}^2},\frac{d}{w_{0}+w_{1}-I}\right[$, we can repeat the process above and obtain  a rational number $\lambda_{3}>1$ such that $(X,\lambda_{3}D)$ is lc in a neighborhood of $[0:1:0:0:0]$ for any $D\in |-K_{X}|$. 
           
           Finally, if we take $\lambda=\min\{\lambda_{1},\lambda_{2},\lambda_{3}\}>1$, we have that the pair $(X,\lambda D)$ is lc for any $D\in |-K_{X}|$. By Lemma 2.4,   $X$ is exceptional.
           \end{itemize}
 \end{proof}


\subsection{Examples}

We exhibit some examples from Propositions 4.1 and 4.2. 

\begin{exm} The following two examples are taken from the list of index 1 of K\"ahler-Einstein 3-folds given in \cite{JK2}.

\begin{enumerate}
    \item Given the weight vector $\mathbf{w}=(66,231,185,259,481)$, we find a quasismooth hypersurface $X:z_{0}^{15}z_{1}+z_{0}z_{1}^{5}+z_{4}z_{2}^{4}+z_{2}z_{3}^{4}+z_{3}z_{4}^{2}=0\subset \mathbb{P}(\mathbf{w})$ of degree $d=1221$. Since the weights $w_{i}$'s and $d$ satisfy the inequality $Id<\frac{4}{3}\min\{ w_{i}w_{j}\}$, then the hypersurface $X$ admits a Kähler-Einstein orbifold metric. On the other hand, the set  $H^{0}(\mathbb{P}(\mathbf{w}),\mathcal{O}(d))$ is generated by the monomials 
    $$\{z_{0}^{15}z_{1}, z_{0}^{8}z_{1}^3, z_{0}z_{1}^{5}, z_{4}z_{2}^{4},  z_{2}z_{3}^{4}, z_{3}z_{4}^{2}\}.$$
 Now, we take the following non-quasismooth hypersurface $$X_{0}:z_{0}^{8}z_{1}^3+z_{0}z_{1}^{5}+z_{4}z_{2}^{4}+z_{2}z_{3}^{4}+z_{3}z_{4}^{2}=0.$$  Clearly, these verify $w_{0}+w_{1}>I=1$ and $d=1221<\min\{w_{i}w_{j}\}=66(185)$. Moreover, the exponents $\alpha_{1}=3$ and $\beta_{0}=1$ divide  $m_{3}=33$ and both are less than  $\frac{d}{w_{0}+w_{1}-1}=\frac{1221}{296}$. Then by Proposition 4.2, the weighted hypersurface $X_{0}$ is exceptional.
 \item We consider the weight vector $\mathbf{w}=(118,118,35,185,135)$. Here, we have the quasismooth hypersurface $X:z_{0}^{5}+z_{0}z_{1}^{5}+z_{4}z_{2}^{13}+z_{2}z_{3}^{3}+z_{3}z_{4}^{3}=0\subset \mathbb{P}(\mathbf{w})$ of degree $d=590$. Easily, we can verify that the weights $w_{i}$'s and $d$ satisfy the inequality $Id<\frac{4}{3}\min\{ w_{i}w_{j}\}$. Thus, the hypersurface $X$ admits a Kähler-Einstein orbifold metric. On the other hand, the set  $H^{0}(\mathbb{P}(\mathbf{w}),\mathcal{O}(d))$ is generated by the monomials 
 $$\{z_{0}^{5},z_{0}^{4}z_{1},z_{0}^3z_{1}^2, z_{0}^2z_{1}^{3},z_{0}z_{1}^{4},z_{1}^{5}, z_{4}z_{2}^{13},  z_{2}z_{3}^{3}, z_{3}z_{4}^{3}\}.$$
 Here, we take the two non-quasismooth hypersurfaces $$X_{0}:z_{0}^{5}+z_{0}^{2}z_{1}^{3}+z_{4}z_{2}^{13}+z_{2}z_{3}^{3}+z_{3}z_{4}^{3}=0.$$ $$X_{1}:z_{0}^{3}z_{1}^{2}+z_{0}^{2}z_{1}^{3}+z_{4}z_{2}^{13}+z_{2}z_{3}^{3}+z_{3}z_{4}^{3}=0.$$
 Next, we will see that $X_{0}$ and $X_{1}$ are exceptional. First, we notice that $w_{0}+w_{1}>I=1$ and $d=590<\min\{w_{i}w_{j}\}=118(35)$.
 \begin{itemize}
     \item For the hypersurface $X_{0}$, we see that the exponent $\beta_{0}=2$ divides  $m_{3}=118$ and it is less than $\frac{d}{w_{0}+w_{1}-1}=\frac{590}{235}$. Thus,  by Proposition 4.1, $X_{0}$ is exceptional.
     \item For the hypersurface $X_{1}$, we have that the exponents $\alpha_{1}=2$ and $\beta_{0}=2$ divide  $m_{3}=118$ and both are less than  $\frac{d}{w_{0}+w_{1}-1}=\frac{590}{235}$. Then by Proposition 4.2, $X_{1}$ is exceptional.
 \end{itemize}
\end{enumerate}
 \end{exm}

\begin{exm}
For the weight vector $\mathbf{w}=(82,82,35,125,95)$, we find a quasismooth hypersurface $X:z_{0}^{5}+z_{0}^5+z_{4}z_{2}^{9}+z_{2}z_{3}^{3}+z_{3}z_{4}^{3}=0 \subset\mathbb{P}(\mathbf{w})$ of degree $d=410$. Here, the index is $I=|\mathbf{w}|-d=9$. Clearly, these verify the inequality $Id<\frac{4}{3}\min\{ w_{i}w_{j}\}$. Thus, the hypersurface $X$ admits a Kähler-Einstein orbifold metric. Now, we consider the generating monomials of the set  $H^{0}(\mathbb{P}(\mathbf{w}),\mathcal{O}(d))$: 
$$\{z_{0}^{5},z_{0}^{4}z_{1},z_{0}^3z_{1}^2, z_{0}^2z_{1}^{3},z_{0}z_{1}^{4},z_{1}^{5}, z_{4}z_{2}^{9},  z_{2}z_{3}^{3}, z_{3}z_{4}^{3}\}.$$ 
Using these monomials, we construct the following hypersurfaces:
$$X_{0}:z_{0}^{5}+z_{0}^{2}z_{1}^{3}+z_{4}z_{2}^{9}+z_{2}z_{3}^{3}+z_{3}z_{4}^{3}=0.$$ $$X_{1}:z_{0}^{4}z_{1}+z_{0}^{2}z_{1}^{3}+z_{4}z_{2}^{9}+z_{2}z_{3}^{3}+z_{3}z_{4}^{3}=0.$$
Let us see that $X_{0}$ and $X_{1}$ are exceptional. First, we notice that $w_{0}+w_{1}>I=9$ and $d=410<\min\{w_{i}w_{j}\}=82(35)$.
 \begin{itemize}
     \item For the hypersurface $X_{0}$, we see that the exponent $\beta_{0}=2$ divides $m_{3}=82$ and it is less than $\frac{d}{w_{0}+w_{1}-I}=\frac{410}{155}$. Then by Proposition 4.1, we have $X_{0}$ is exceptional.
     \item For the hypersurface $X_{1}$, we have that the exponents $\alpha_{1}=1$ and $\beta_{0}=2$ divide  $m_{3}=82$ and both are less than  $\frac{d}{w_{0}+w_{1}-I}=\frac{410}{155}$. Then by Proposition 4.2, the hypersurface $X_{1}$ is exceptional.
 \end{itemize}
\end{exm}

\begin{exm}
    We consider the weight vector $\mathbf{w}$ of the form
    \begin{equation}
        \mathbf{w}=\left(9k+1,9k+1,5(3k-2),5(2k+1),35\right),
    \end{equation}
    where $k\in\{5,7,9,13\}$. For each $k$, we take the following three types of non-quasismooth hypersurfaces of degree $d=5(9k+1)$:
    \begin{align*}    X_{0} & :z_{0}^{5}+z_{0}^{2}z_{1}^{3}+z_{4}z_{2}^{3}+z_{2}z_{3}^{3}+z_{3}z_{4}^{k}=0.\\
    X_{1} & :z_{0}^{4}z_{1}+z_{0}^{2}z_{1}^{3}+z_{4}z_{2}^{3}+z_{2}z_{3}^{3}+z_{3}z_{4}^{k}=0.\\
    X_{2} & : z_{0}^{3}z_{1}^2+z_{0}^{2}z_{1}^{3}+z_{4}z_{2}^{3}+z_{2}z_{3}^{3}+z_{3}z_{4}^{k}=0.
    \end{align*}
    Let us see that these are exceptional. Computing the index, we obtain $I=27-2k>0$. Then
    $$w_{0}+w_{1}=2(9k+1)>I=27-2k.$$
    In addition, we have that $\min\{ w_{i}w_{j}\}=35(9k+1)>d$. 
    \begin{itemize}
     \item For the hypersurface $X_{0}$, we see that the exponent $\beta_{0}=2$ divides  $m_{3}=9k+1$, which is even. Moreover, it is less than $\frac{d}{w_{0}+w_{1}-I}=\frac{5(9k+1)}{20k-25}$. Then by Proposition 4.1, we have $X_{0}$ is exceptional.
     \item For the hypersurface $X_{1}$, we have that the exponents $\alpha_{1}=1$ and $\beta_{0}=2$ divide  $m_{3}=9k+1$ and both are less than  $\frac{d}{w_{0}+w_{1}-I}=\frac{5(9k+1)}{20k-25}$. Then by Proposition 4.2, the hypersurface $X_{1}$ is exceptional.
     \item For the hypersurface $X_{2}$, we have that the exponents $\alpha_{1}=\beta_{0}=2$ divide $m_{3}=9k+1$ and both are less than $\frac{d}{w_{0}+w_{1}-I}=\frac{5(9k+1)}{20k-25}$. Then by Proposition 4.2, the hypersurface $X_{2}$ is also exceptional.
 \end{itemize}
\end{exm}

\section*{Declarations}





\subsection*{Funding} The first author received  financial support from Pontificia Universidad Católica del Perú through project VRI-DFI 2019-1-0089.

\subsection*{Data Availability} Data sharing not applicable to this article as no datasets were generated or analysed during the current study.

\subsection*{Conflict of interests} The authors have no competing interests to declare that are relevant to the content of this article.




\begin{thebibliography}{9}





\bibitem{ATW}
Abramovich, D.,  Temkin, M.,  Włodarczyk, J., \emph{Functorial embedded resolution via weighted blowings up,} Algebra \& Number Theory, Vol. 18, 1557-1587, (2024).  
DOI: 10.2140/ant.2024.18.1557


\bibitem{Bi2}
 Birkar, C., \emph {Singularities of linear systems and boundedness of Fano varieties,}  Ann. of Math. 193, no. 2, 347-405., (2021).


\bibitem{BLX}
Blum, H.,  Liu, Y.,  Xu, C.,  \emph{Openness of $K$-semistability for Fano varieties,} Duke Math. J. 171, 2753-2797, (2022).














\bibitem{BG1}
Boyer, C.P., Galicki, K. \emph{New Einstein metrics in dimension five,} J. Differential Geom. 57, no. 3, 443-463, (2001).



\bibitem{BBG}
Boyer, C.P., Galicki, K. \emph{Sasakian Geometry,} Oxford University Press, (2008).




\bibitem{BGN1}
Boyer, C.P., Galicki, K., Nakamaye, M. \emph{Einstein Metrics on Rational Homology 7-Spheres,} Ann.Inst.Fourier 52, no.5, 1569-1584, (2002).

\bibitem{BGN2}
Boyer, C.P., Galicki, K., Nakamaye, M. \emph{On the Geometry of Sasakian-Einstein 5-Manifold,} Math. Ann. 325, 485-524, (2003).



\bibitem{BGK}
Boyer, C.P., Galicki, K., Koll\'ar, J. \emph{Einstein Metrics on Spheres,} Annals of Mathematics, 162 , 557-580, (2005).


\bibitem{CL}
Cuadros, J., Lope J., \emph{Sasaki-Einstein 7-manifolds and Orlik's conjecture,} Ann. Glob. Anal. Geom. 65, 3, (2024). 




\bibitem{CL2}
Cuadros, J., Lope J., \emph{The local moduli of Sasaki-Einstein rational homology 7-spheres and invertible polynomials,} 
https://arxiv.org/abs/2503.18650, (2025).


\bibitem{CS}
Cheltsov, I.A., Shramov, K.A. \emph{Log canonical thresholds of smooth Fano threefolds} Russian Math. Surveys, 63: 5 859-959, (2008). 

\bibitem{CDS}
Chen, X. X., Donaldson, S. K., Sun, S. \emph{Kähler-Einstein metrics on Fano manifolds, I-III.} J. Amer. Math. Soc. 28, 183-197, 199-234, 235-278, (2015).

\bibitem{DK}
Demailly, J.-P., Kollár, J.,  \emph{Semi-continuity of complex singularity exponents and Kähler-Einstein metrics on Fano orbifolds,} Ann. Sci. École Norm. Sup. 34, 525-556, (2001).

\bibitem{Do}
 Donaldson, S.K.,  \emph{Scalar curvature and stability of toric varieties,} J. Differential Geom., 62(2):289-349, (2002).

\bibitem{IF}
Iano-Fletcher, A. R., \emph{Working with weighted complete intersections,} Explicit birational geometry of 3-folds, 101-173. London Math. Soc. Lecture Notes Ser. 281, Cambridge Univ. Press, Cambridge (2000).


\bibitem{IP}
Ishii, S., Prokhorov, Y.,  \emph{Hypersurface exceptional singularities,} Internat. J. Math. 12, 661-687, (2001).


\bibitem{JK1}
Johnson, J.M., Kollár, J., \emph{Kähler-Einstein metrics on log del Pezzo surfaces in weighted projective 3-spaces,} 
Ann. Inst. Fourier (Grenoble) 51, 69-79, (2001).

\bibitem{JK2}
Johnson, J.M., Kollár, J.,  \emph{Fano Hypersurfaces in Weighted Projective 4-Space,} Exper. Math, 10, no.1, 151-158, (2004).










\bibitem{KS}
Kreuzer, M.,  Skarke, H., \emph{On the classification of quasihomogeneous functions,} Comm. Math. Phys. 150  no. 1, 137-147, (1992).






\bibitem{Ko1}
Kollár, J.,  \emph{Singularities of pairs,} Algebraic geometry (Santa Cruz, 1995), 221-287. Proc. Symp. Pure Math. 62, Amer. Math. Soc. (1997). 

\bibitem{KM}
Kollár, J., Mori, S., \emph{Birational geometry of algebraic varieties,} Cambridge Tracts in Mathematics, vol. 134, Cambridge University Press, Cambridge, (1998). 

\bibitem{Kol1}
 Kollár, J.,  \emph{Einstein metrics on five-dimensional Seifert bundles,}  J. Geom. Anal. 15, no. 3, 445-476, (2005).


\bibitem{Kol2}
Kollár, J.,  \emph{Einstein Metrics on connected sums of $S^2 \times S^3$,} J. Diff. Geom. 77, no. $2,259-272$, (2007).

\bibitem{Ko2}
Kollár J., \emph{Is There a Topological Bogomolov-Miyaoka-Yau Inequality?,} Pure and Applied Mathematics Quarterly Volume 4, Number 2 (Special Issue: In honor of Fedor Bogomolov, Part 1 of 2) 203-236, (2008).

\bibitem{Ko3}
Kollár, J.,  \emph{ Singularities of the minimal model program,} With the collaboration of Sándor Kovács. Cambridge, (2013).
 
 \bibitem{Ko4}
Kollár, J.,  \emph{Links of complex analytic singularities,} Surveys in Differential Geometry, Volume 18, 157-193 (2013).

\bibitem{Li}
Li, C.  \emph{Notes on weighted Kähler-Ricci solitons and application to Ricci-flat Kähler cone metrics,}  arXiv:2107.02088 (2021).

 \bibitem{LXZ}
Liu, Y.,  Xu, C., Zhuang, Z., \emph{Finite generation for valuations computing stability thresholds and applications to $K$-stability,} Ann. of Math. 196, 507-566, (2022).





\bibitem{Mc}
McQuillan, M. \emph{Very functorial, very fast, and very easy resolution of singularities,} Geom. Funct. Anal. 30, 858–909 (2020). https://doi.org/10.1007/s00039-020-00523-7

\bibitem{Mi}
Milnor, J. \emph{Singular Points of Complex Hypersurfaces,} Annals of Mathematical Studies, Vol.61, Princeton
University Press, Princeton, NJ, (1968).

\bibitem{MO}
Milnor, J., Orlik, P. \emph{Isolated Singularities defined by Weighted Homogeneous Polynomials,} Topology 9,  385-393, (1970).


\bibitem{Od}
Odaka, Y.,  \emph{Compact moduli of Calabi-Yau cones and Sasaki-Einstein spaces,} https://arxiv.org/abs/2405.07939v3, (2024).

\bibitem{OS}
Odaka Y.,   Sano, Y.,  \emph{Alpha invariant and K-stability of $\mathbb{Q}$-Fano varieties,} Adv. Math. 229, no. 5, 2818-2834, (2012). 

\bibitem{Or}
Orlik, P. \emph{On the Homology of Weighted Homogeneous Manifolds,} Proceedings of the Second Conference on Compact Transformation Groups (Univ. Mass, Amherst, Mass 1971) Part I (Berlin), Spring,  pp 260-269, (1972).

\bibitem{OW}
Orlik, P., Wagreich, P. \emph{Seifert n-manifolds,} Invent. Math. 28, 137-159. MR (1975).








\bibitem{Sh}
Shokurov, V.V, \emph{Birational Geometry of Algebraic Varieties, Open Problems list}  from the 23rd International Symposium of the Division of Mathematics of the Taniguchi Foundation, Katata, Japan,  30-32, (1988).




\bibitem{Ti}
Tian, G.,  \emph{On Kähler-Einstein metrics on certain Kähler manifolds with $c_1(M)>0$,} Invent. Math. 89, 225-246, (1987). 

\bibitem{Ti2}
Tian, G.,  \emph{K-stability and Kähler-Einstein metrics,}
Comm. Pure Appl. Math. 68, no. 7, 1085–1156, (2015).

\bibitem{To}
Totaro, T., \emph{Klt varieties with conjecturally minimal volume,} Int. Math. Res. Not. IMRN, no. 1, 462-491, (2024).


\end{thebibliography}
\end{document}